%% file: main.tex
\documentclass[11pt,a4paper]{article}

\usepackage{a4wide}
\usepackage{graphicx}
\usepackage{latexsym}
\usepackage{epsfig}
\usepackage{amssymb}
\usepackage{amstext}
\usepackage{amsgen}
\usepackage{amsxtra}
\usepackage{amsgen}
\usepackage{amsthm}
\usepackage{enumerate,amsmath,subfigure,color}
\usepackage{tikz}
\usepackage{pgfplots}
\usepackage{todonotes}

\newtheorem{thm}{Theorem}[section]
\newtheorem{prop}[thm]{Proposition}
\newtheorem{lemma}[thm]{Lemma}
\newtheorem{cor}[thm]{Corollary}
\newtheorem{rem}[thm]{Remark}

\theoremstyle{definition}
\newtheorem{example}[thm]{Example}
\newtheorem{definition}[thm]{Definition}

\numberwithin{equation}{section}

\newcommand{\N}{\mathbb{N}}

\newcommand{\R}{\mathbb{R}}
\newcommand{\dom}{{\mathrm{dom}}}
\newcommand{\norm}[1]{\left\Vert #1 \right\Vert}
\newcommand{\abs}[1]{\left\vert #1 \right\vert}

\renewcommand{\H}{{\cal H}}
\newcommand{\X}{{\cal X}}

\renewcommand{\div}{\operatorname{div}}
\newcommand{\divw}{\mathrm{div}_w\,}
\newcommand{\nablaw}{\nabla_w}

\newcommand{\argmin}{\mathrm{arg}\min}
\newcommand{\argmax}{\mathrm{arg}\max}
\newcommand{\st}{\,:\,}

\newcommand{\supp}{\mathrm{supp}}
\newcommand{\dx}{\,\mathrm{d}x}
\renewcommand{\d}{\,\mathrm{d}}

\newcommand{\calN}{\mathcal{N}}
\newcommand{\calM}{\mathcal{M}}

\newcommand{\calJ}{\mathcal{J}}

\newcommand{\sgn}{\mathrm{sgn}}

\newcommand{\eps}{\varepsilon}

\newcommand{\sg}{\zeta}

\newcommand{\dist}{\mathrm{dist}}
\newcommand{\extr}{\operatorname{extr}}
\newcommand{\inn}{\operatorname{int}}

\pgfplotsset{compat=1.14}

\begin{document}

\title{Structural analysis of an $L$-infinity variational problem and relations to distance functions}
\author{Leon Bungert\thanks{Department Mathematik, Universit\"{a}t Erlangen-N\"{u}rnberg, Cauerstrasse 11, 91058 Erlangen, Germany. \texttt{\{leon.bungert,martin.burger\}@fau.de}} 
\and Yury Korolev\thanks{Department of Applied Mathematics and Theoretical Physics, University of Cambridge, Wilberforce Road, Cambridge CB3 0WA, UK. \texttt{y.korolev@damtp.cam.ac.uk}}
\and Martin Burger\footnotemark[1] }
\maketitle

\begin{abstract}
In this work we analyse the functional $\calJ(u)=\norm{\nabla u}_\infty$ defined on Lipschitz functions with homogeneous Dirichlet boundary conditions. 
Our analysis is performed directly on the functional without the need to approximate with smooth $p$-norms. 
We prove that its ground states coincide with multiples of the distance function to the boundary of the domain. Furthermore, we compute the $L^2$-subdifferential of $\calJ$ and characterize the distance function as unique non-negative eigenfunction of the subdifferential operator.
We also study properties of general eigenfunctions, in particular their nodal sets.
Furthermore, we prove that the distance function can be computed as asymptotic profile of the gradient flow of $\calJ$ and construct analytic solutions of fast marching type. 
In addition, we give a geometric characterization of the extreme points of the unit ball of $\calJ$.

Finally, we transfer many of these results to a discrete version of the functional defined on a finite weighted graph. Here, we analyze properties of distance functions on graphs and their gradients. The main difference between the continuum and discrete setting is that the distance function is not the unique non-negative eigenfunction on a graph.


\end{abstract}

\tableofcontents

\section{Introduction}
\subsection{Eigenvalue problems associated to Rayleigh quotients}
Eigenvalue problems are a very old tool in mathematics with a long list of theoretical and practical applications. In particular, nonlinear eigenvalue problems have become increasingly popular in the last decades due to their challenging mathematical properties and their wide range of theoretical and practical applications. A special class of nonlinear eigenvalue problems are those which arise from a variational principle, like the minimization of a Rayleigh quotient 
\begin{align}
    \frac{J(u)}{H(u)}\to\min,
\end{align}
where $J$ and $H$ typically are convex functionals which share the same homogeneity. In this abstract setting the eigenvalue problem is often defined by
\begin{align}\label{eq:general_eigenvalue_problem}
    \lambda\partial H(u)\cap\partial J(u)\neq\emptyset,
\end{align}
where $\lambda=J(u)/H(u)$ denotes the eigenvalue and $\partial$ stands for the subdifferential. For smooth $J$ and $H$ this is exactly the condition for being a critical point of the Rayleigh quotient. Elements actually minimizing the Rayleigh quotient, and thus having the lowest possible eigenvalue, are referred to as ground states. Obviously, due to the homogeneity of $J$ and $H$ ground states are invariant under multiplication with a scalar. By choosing
\begin{align}
    J(u)=\int_\Omega|\nabla u|^p\dx,\qquad
    H(u)=\int_\Omega|u|^p\dx,
\end{align}
one obtains the eigenvalue problem of the $p$-Laplacian
\begin{align}
    \lambda|u|^{p-2}u=-\div(|\nabla u|^{p-2}\nabla u),
\end{align}
which has to be complemented with suitable boundary conditions, and is a very well-studied nonlinear eigenvalue problem (see, for instance, \cite{binding2006basis,kawohl2006positive,barles1988remarks,le2006eigenvalue,kawohl2008p}). Interesting but challenging limit cases are $p\to 1$ and $p\to\infty$ since in these cases functionals $J$ and $H$ are non-smooth and not strictly convex. In particular, this means that there can exist linearly independent ground states. For more details about the 1-Laplacian eigenvalue problem we refer to \cite{kawohl2007dirichlet}, explicit solutions can be found in \cite{bellettini2005explicit,alter2005characterization}. The infinity-Laplacian eigenvalue equation takes the form
\begin{align}\label{eq:infinity_eigen}
    0 = 
    \begin{cases}
        \min(|\nabla u|-\lambda u,-\Delta_\infty u),\quad&u>0,\\
        -\Delta_\infty u,\quad&u=0,\\
        \max(-|\nabla u|-\lambda u,-\Delta_\infty u),\quad&u<0,\\
    \end{cases}    
\end{align}
which has to be understood in the viscosity sense. Typically, the problem is complemented with homogeneous Dirichlet conditions. We refer to \cite{juutinen1999eigenvalue,juutinen1999infinity,yu2007some} for more details. Positive solutions of \eqref{eq:infinity_eigen} on a domain $\Omega$ are called infinity ground states and indeed they minimize the Rayleigh quotient
\begin{align}\label{eq:infinity_groundstate}
    u\mapsto\frac{\norm{\nabla u}_\infty}{\norm{u}_\infty}
\end{align}
among all functions $u\in W^{1,\infty}(\Omega)$ that vanish on the boundary $\partial\Omega$. However, due to the lack of strict convexity, minimizers of \eqref{eq:infinity_groundstate} are far from being unique up to scalar multiplication. In particular, the distance function $x\mapsto\dist(x,\partial\Omega)$ is always a minimizer of \eqref{eq:infinity_groundstate} but not necessarily a solution of \eqref{eq:infinity_eigen}. Furthermore, also solutions of \eqref{eq:infinity_eigen} are not unique \cite{hynd2013nonuniqueness}. The infinity-Laplacian eigenvalue problem falls under the scope of  $L^\infty$-variational problems which have been an active field of research, with the main contributions being due to Aronsson (see~\cite{aronsson2004tour} for an overview). One big challenge with these problems is that the involved subdifferentials lie in a space of measures and not in a function space. 

\subsection{Structure of regularizers}

From an application point of view, eigenvalue problems of the form \eqref{eq:general_eigenvalue_problem} are interesting since they allow to study the structural properties of the functional $J$, if it is interpreted as regularization functional. 
For instance, in the case of $J:\H\to\R\cap\{\infty\}$ being defined on a Hilbert space $\H$, and $H(\cdot)=\norm{\cdot}_\H$ coinciding with its norm, it holds that eigenfunctions $f$ are precisely the separated variables solutions to the gradient flow
\begin{align}\label{gradflow}
    \begin{cases}
    u'(t)+\partial J(u(t))\ni 0,\\
    u(0)=f,
    \end{cases}
\end{align}
In this case the solution of \eqref{gradflow} has the form $u(t)=a(t)f$ where function $a(t)$ depends on the homogeneity of $J$ (cf.~\cite{bungert2019asymptotic,bungert2019nonlinear,burger2016spectral,cohen2020introducing}).
If $J$ is one-homogeneous and $f$ is an eigenfunction, then this separated variable solution also solves the variational regularization problem 
\begin{align}\label{varprob}
    \frac{1}{2}\norm{u-f}^2_\H+tJ(u).
\end{align}
Recent results for general homogeneous functionals~\cite{bungert2019asymptotic,bungert2019nonlinear} showed that also for general data $f$, the gradient flow \eqref{gradflow} behaves like a separate variable solution asymptotically.
Under some conditions it was shown that asymptotic profiles of \eqref{gradflow} are eigenfunctions, meaning
\begin{align}\label{eq:asymptotic_profiles}
    \lim_{t\to\infty}\frac{u(t)}{\norm{u(t)}_\H}=w,\qquad
    \lim_{t\to\infty}\frac{J(u(t))}{\norm{u(t)}_\H}=\lambda,\qquad
    \lambda\frac{w}{\norm{w}_\H}\in\partial J(w).
\end{align}
Subsuming these results, one can say that eigenfunctions to some extend describe which structures are preserved by regularization methods like \eqref{gradflow} or \eqref{varprob}.
For example, in the case of $J$ being the total variation, it is well-known that a large class of eigenfunctions are given by so-called calibrable sets \cite{alter2005characterization}, which provides an explanation of the staircasing effect in total variation regularization \cite{burger2013guide}. Furthermore, the study of regularizers through their eigenfunction has sparked applications in image processing, as for instance in \cite{gilboa2014total,benning2017nonlinear}.

An alternative way to study structural properties of regularizers is through the extreme points of their unit ball, where the extreme points of a convex set $C$ in a vector space are given by
\begin{align}
    \extr(C):=\left\lbrace u\in C\st \nexists\,v\neq w\in C,\,\lambda\in(0,1)\st u=\lambda v+(1-\lambda)w\right\rbrace.
\end{align}
So-called representer theorems study qualitative properties of solutions to the optimization problems
\begin{subequations}\label{eq:prob_repr_thm}
 \begin{align}
    &u^*\in\argmin_{u\in \X}J(u)\st Au=f,\\
    \text{or}\quad&u^*\in\argmin_{u\in\X} F(Au)+J(u),
\end{align}
\end{subequations}
where $\X$ is a Banach space and $A:\X\to\H$ is a linear operator mapping into a \emph{finite-dimensional} Hilbert space. The functionals $J$ and $F$ are convex regularization and data fitting functionals, respectively. 
Recent results \cite{bredies2018sparsity,boyer2019representer,unser2019unifying} show that in this case there exists a minimizer $u^*$ of \eqref{eq:prob_repr_thm} which can essentially be expressed as finite linear combination of extreme points in the unit ball of $J$, meaning
\begin{align}
    u^*=n+\sum_{i=1}^k c_iu_i,
\end{align}
where $n\in\calN(J)$ denotes an element in the null-space of $J$, $(c_i)$ are real numbers, and $(u_i)\subset\extr(B_J)$ are extreme points of the unit ball $B_J=\{u\in\X\st J(u)\leq 1\}$.
Typically, extreme points have interesting geometric properties which they hand down to minimizers of~\eqref{eq:prob_repr_thm}. If $J$ equals the total variation of a function, for instance, extreme points are given by characteristic functions of so-called simple sets \cite{bredies2018sparsity}, which gives yet another explanation for the staircasing phenomenon.

\subsection{Set-up and outline of this paper}
Let $\Omega\subset\R^n$ be an open and bounded domain and for $1\leq p\leq\infty$ we let $\norm{\cdot}_p$ denote the Lebesgue $p$-norms of functions or vector fields. We define the function space
\begin{align}
    W^{1,\infty}_0(\Omega):=\{u\in W^{1,\infty}(\Omega)\st u=0\text{ on }\partial\Omega\}
\end{align}
which consists of all Lipschitz continuous functions, vanishing on $\partial\Omega$.
In this paper we study the functional 
\begin{align}\label{eq:fctl_v2}
\calJ(u)=
\begin{cases}
\norm{\nabla u}_\infty,\quad &u\in W^{1,\infty}_0(\Omega),\\
+\infty,\quad &u\in L^2(\Omega)\setminus W^{1,\infty}_0(\Omega),
\end{cases}
\end{align}
which coincides with the Lipschitz constant if $u\in W^{1,\infty}_0(\Omega)$.
We would like to understand its structure in terms of eigenfunctions and extreme points.
\begin{rem}\label{rem:lipschitz}
Although the space $W^{1,\infty}(\Omega)$ only coincides with the Lipschitz functions on $\Omega$ if $\Omega$ is at least quasi-convex \cite{heinonen2005lectures}, for the space $W^{1,\infty}_0(\Omega)$ this is always true.
Furthermore, $\calJ(u)$ equals the Lipschitz constant of $u\in W^{1,\infty}_0(\Omega)$.
This is due to the fact that functions in $W^{1,\infty}_0(\Omega)$ can be extended by zero to lie in $W^{1,\infty}(\R^n)$, which coincides with the space of all Lipschitz functions due to the convexity of $\R^n$. 
\end{rem}

Although $\calJ$ is defined on $L^2(\Omega)$ and hence admits standard Hilbert space subdifferential calculus, it comes with many of the challenges and properties of a pure $L^\infty$-variational problem. The associated Rayleigh quotient is 
\begin{align}\label{eq:rayleigh_calJ}
    u\mapsto\frac{\calJ(u)}{\norm{u}_2}=\frac{\norm{\nabla u}_\infty}{\norm{u}_2},
\end{align}
and admits an easier treatment than the ``pure'' $L^\infty$ Rayleigh quotient \eqref{eq:infinity_groundstate} due to the presence of the $L^2$-norm in the denominator.
In particular, \eqref{eq:rayleigh_calJ} has essentially a unique minimizer, given by the distance function to the boundary of the domain.
Note that a similar functional has been studied in \cite{burger2015infimal} and a Rayleigh quotient of mixed $L^\infty$-$L^2$-type was considered in~\cite{barron2005minimizing}.
While in the first work the analysis is limited to the one-dimensional case, and in the second work the authors approximate the $L^\infty$-norm with smooth $p$-norms, our subdifferential techniques work in arbitrary dimension and without approximation.
The abstract eigenvalue problem \eqref{eq:general_eigenvalue_problem} associated to $\calJ$ becomes
\begin{align}\label{eq:eigenvalue_problem_calJ}
    \lambda \frac{u}{\norm{u}_2}\in\partial\calJ(u).
\end{align}

We also consider a discrete variant of $\calJ$ defined on a finite weighted graph and transfer most of our continuous results to the discrete setting. Naturally, due to the finite dimensional character of graphs, the proofs simplify a lot. However, the non-local nature of graphs makes the results interesting, nevertheless. 
In particular, the ground state of this functional is also given by the distance function with respect to a the weighted graph distance. From an applied point of view, this interpretation as nonlinear eigenfunction opens the doors for new computational methods for the distance function on graphs. Traditional approaches to compute distance functions on graphs or grids typically rely on level set methods or schemes to solve the Eikonal equation $|\nabla u|=1$, see for instance \cite{memoli2001fast,desquesnes2010efficient,desquesnes2013eikonal}. Although this paper is mainly of theoretical nature, in Figure~\ref{fig:distance_fcts} we show some distance functions on graphs which were computed using asymptotic profiles of gradient flows in the sense of \eqref{eq:asymptotic_profiles}, see also~\cite{bungert2019nonlinear,bungert2019asymptotic,bungert2019computing} for theory and computational results for the 1-Laplacian on graphs, respectively.
\begin{center}
\begin{figure}
    \centering
    \begin{minipage}[t]{0.48\textwidth}
        \vspace*{0cm}
        {\includegraphics[width=0.6\textwidth,trim=4cm 4cm 4cm 4cm, clip]{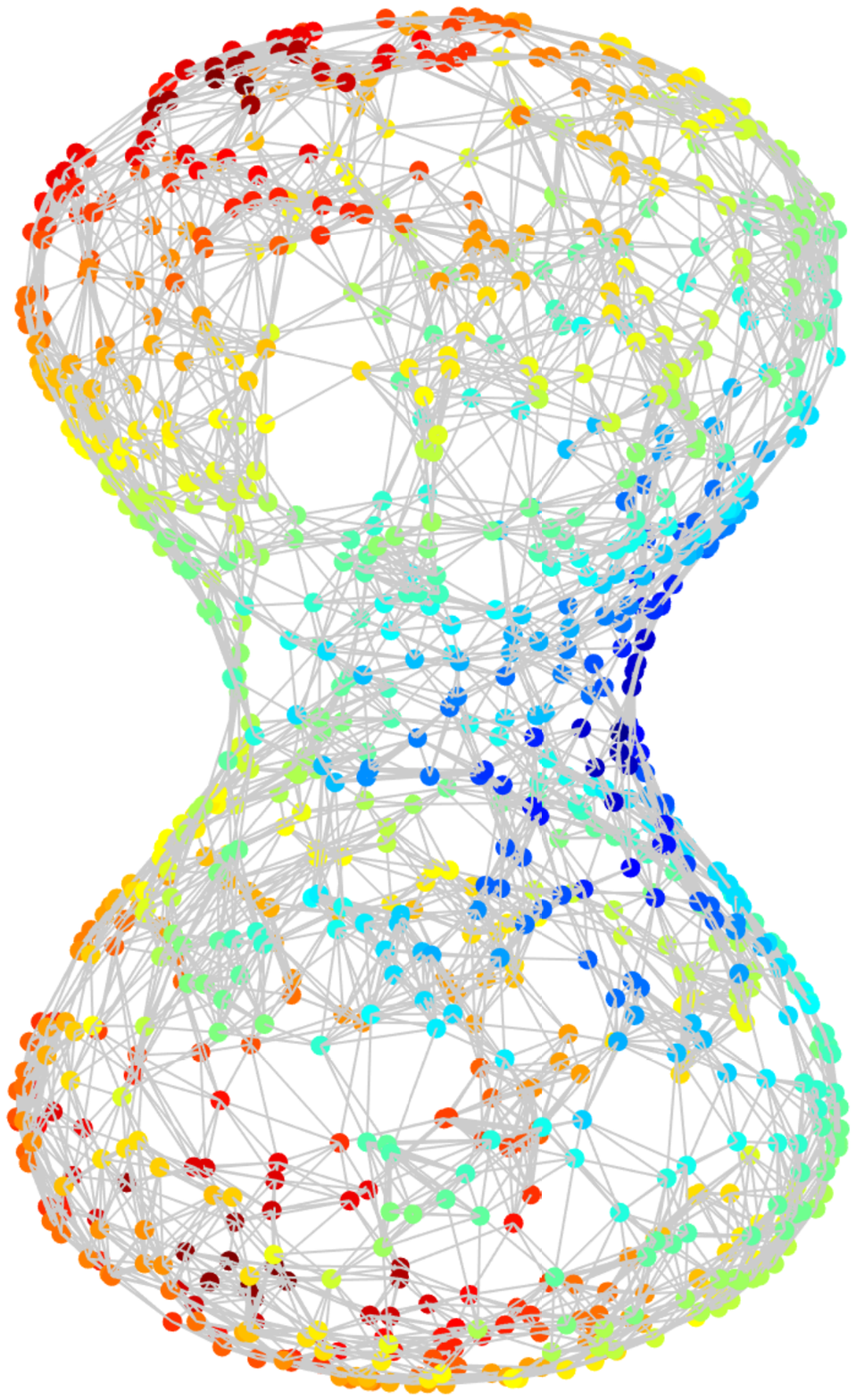}}
    \end{minipage}  
    \hfill
    \begin{minipage}[t]{0.48\textwidth}
        \vspace*{0cm}
        {\includegraphics[width=0.7\textwidth,trim=6cm 9cm 5cm 8cm, clip]{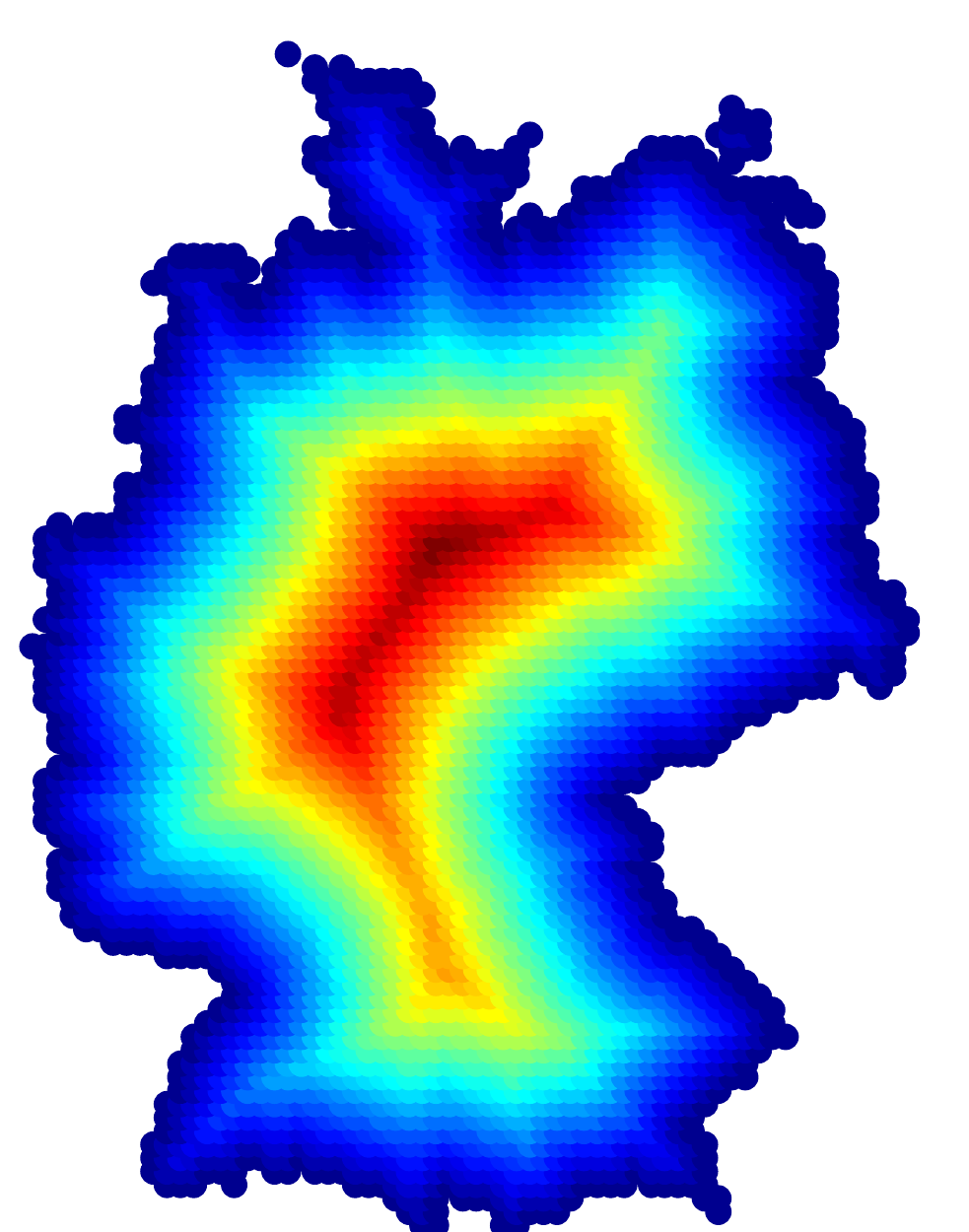}}
    \end{minipage}  
    \caption{\textbf{Left:} distance function to a point on a discretized manifold, \textbf{right:} distance function to the boundary of a grid graph}
    \label{fig:distance_fcts}
\end{figure}
\end{center}%

This paper is organized as follows. In Section~\ref{sec:analysis_fctl} we analyze spectral properties of the functional $\calJ$.
We characterize ground states as distance functions and compute the $L^2$-subdifferential of in Sections~\ref{sec:ground_states} and~\ref{sec:subdifferential}, respectively.
Subsequently, in Section~\ref{sec:eigenfunctions} we study the geometrical properties of eigenfunctions. 
In particular, we prove that under a regularity condition, the nodal set of eigenfunctions has zero Lebesgue measure.
Next, in Section~\ref{sec:explicit_sol} we construct an explicit solution to the gradient flow and variational regularization problem of $\calJ$ which converges to the distance function and possesses level sets that move parallelly to the boundary of the domain.
In Section~\ref{sec:extreme_points} we give a characterization of the extreme points of the unit ball, which gives intuition on the geometrical structure of optimization problems involving~$\calJ$.
In Section~\ref{sec:graph} we transfer most of these results to finite weighted graphs. We prove that ground states are distance functions in Section~\ref{sec:graph_dist} and study some properties of graph distance functions. In Section~\ref{sec:graph_properties} we finally collect the graph versions of our results from Sections~\ref{sec:analysis_fctl} and~\ref{sec:extreme_points}, hereby skipping most of the proofs since they are elementary, given the proofs in the continuous setting.

We would like to conclude with a remark on how to read this paper. For those readers who are primarily interested in graphs, it is possible to only read Section~\ref{sec:graph} since it is self-contained in its presentation. Similarly, readers interested mainly in the continuous setting are welcome to only read Section~\ref{sec:analysis_fctl} since the results in the graph setting are somewhat similar. 

\section{Spectral properties}
\label{sec:analysis_fctl}
\subsection{Ground states}
\label{sec:ground_states}
In this section we will investigate the ground states of $\calJ$, i.e., minimizers of the nonlinear Rayleigh quotient
\begin{align}\label{eq:ground_state}
u^\ast\in\argmin_{u\in W^{1,\infty}_0(\Omega)}\frac{\calJ(u)}{\norm{u}_2}.
\end{align}
We prove that---up to multiplicative constants---they coincide with the distance function of the boundary $\partial\Omega$ of the domain which is defined as
\begin{align}\label{eq:distance_fct}
d(x):=\dist(x,\partial\Omega):=\inf_{y\in\partial\Omega}|x-y|.
\end{align}
Note that this in particular implies that ground states are unique up to scaling, which is often referred to as simplicity. Indeed, our statement is slightly more general since it holds for minimizers of 
\begin{align}\label{eq:p_ground_state}
u^\ast\in\argmin_{u\in W^{1,\infty}_0(\Omega)}\frac{\calJ(u)}{\norm{u}_p},\quad 1\leq p<\infty,
\end{align}
where \eqref{eq:ground_state} is a special case when choosing $p=2$.

\begin{thm}[Ground states are distance functions]\label{thm:ground_states_distance}
All solutions $u^\ast$ to \eqref{eq:p_ground_state} are multiples of the distance function to $\partial\Omega$, given by \eqref{eq:distance_fct}.
\end{thm}
\begin{proof}
By homogeneity, the solutions to \eqref{eq:p_ground_state} are given by multiples of the solutions to 
\begin{align*}
\hat{u}\in\;&\argmax\left\lbrace{\norm{u}_p}\st\calJ(u)=1\right\rbrace\\
=\;&\argmax\left\lbrace{\norm{u}_p}\st |\nabla u|\leq 1\text{ a.e. in }\Omega,\,u\vert_{\partial\Omega}=0\right\rbrace.
\end{align*}
From \cite{zagatti2014maximal} we infer that---up to global sign---$\hat{u}$ coincides with the unique viscosity solution of the eikonal equation which is given by the distance function~\eqref{eq:distance_fct}.
\end{proof}

Hence, we have characterized the distance function to the boundary of a set in $\R^n$---whose properties are well-known and have been investigated for decades already---as solution to an nonlinear eigenvalue problem associated to the nonlinear and multi-valued operator $\partial\calJ$. As already mentioned in the introduction, it is important to notice the difference between our model and infinity Laplacian ground states (cf.~\cite{juutinen1999eigenvalue,barron2008infinity} for an overview), which are defined as positive viscosity solutions to 
\begin{align}
    \min\left\lbrace|\nabla u|-\Lambda_\infty u,-\Delta_\infty u\right\rbrace=0,
\end{align}
where $\Delta_\infty$ denotes the infinity Laplacian. Here, the eigenvalue $\Lambda_\infty$ is given by
\begin{align}\label{eq:infinity_princ_ev}
    \Lambda_\infty:=\min_{u\in W^{1,\infty}_0(\Omega)}\frac{\norm{\nabla u}_\infty}{\norm{u}_\infty}=\frac{1}{\max_{x\in\Omega}\dist(x,\partial\Omega)}
\end{align}
and every infinity ground state realizes the minimum. However, also the distance function is a minimizer but \emph{no} infinity ground state, in general \cite{juutinen1999infinity}, which means that there are minimizers of \eqref{eq:p_ground_state} for $p=\infty$ which are no multiple of the distance function.

\subsection{Subdifferential}
\label{sec:subdifferential}
In the following we would like to characterize the $L^2$-subdifferential of functional $\calJ$,
which is given by
\begin{align}\label{eq:subdifferential}
    \partial\calJ(u)=\left\lbrace\sg\in L^2(\Omega)\st \langle\sg,v\rangle\leq\calJ(v),\;\forall v\in L^2(\Omega),\;\langle\sg,u\rangle=\calJ(u)\right\rbrace,\quad u\in L^2(\Omega),
\end{align}
since $\calJ$ is absolutely one-homogeneous (cf.~\cite{benning2013ground,burger2016spectral, bungert2019asymptotic,bungert2019nonlinear}, for instance). 
Note that the $L^2$-subdifferential of the functionals
\begin{align}
    \calJ_p(u)=\norm{\nabla u}_p,\quad 1<p<\infty,
\end{align}
is single-valued for $u\in W^{1,p}_0(\Omega)\setminus\{0\}$ and given by
\begin{align}\label{eq:p-laplacians}
    \partial\calJ_p(u)=-\calJ_p(u)^{1-p}\Delta_pu,
\end{align}
where $\Delta_pu:=\div(|\nabla u|^{p-2}\nabla u)$ denotes the $p$-Laplacian. Hence, one could think that by sending $p\to\infty$ one obtains an expression for the subdifferential of $\calJ$ which involves the $\infty$-Laplacian. This, however, turns out not to be the case since the competing limits in~\eqref{eq:p-laplacians} lead to a loss of regularity, as we will see below.

To formulate the subdifferential we define the space
\begin{align}
H(\div;\Omega)&:=\left\lbrace q\in L^2(\Omega)\st\div q\in L^2(\Omega)\right\rbrace
\end{align}
of all $L^2$-vector-fields whose distributional divergence is square-integrable.
The space $H(\div;\Omega)$ is a Hilbert space when equipped with the inner product
\begin{align}
    \langle q,r\rangle_{H(\div;\Omega)}=\int_\Omega \left[q\cdot r+(\div q)(\div r)\right]\dx.
\end{align}
\begin{rem}
It is well-known that vector fields in $H(\div;\Omega)$ posses a normal trace and furthermore the space $C^\infty(\overline{\Omega},\R^n)$ of smooth vector fields is dense in $H(\div;\Omega)$, see for instance \cite[Ch.~1]{girault2012finite}.
\end{rem}
Using that $W^{1,\infty}_0(\Omega)\subset H^1_0(\Omega)$ one obtains the following integration by parts formula, which we will use throughout this work without further references.
\begin{prop}[Integration by parts]\label{prop:integration_by_parts}
Let $q\in H(\div;\Omega)$ and $u\in W^{1,\infty}_0(\Omega)$.
Then it holds
\begin{align}
    \int_\Omega -(\div q) u\dx = \int_\Omega q\cdot\nabla u\dx.
\end{align}
\end{prop}
The following closed subspace of $H(\div;\Omega)$---which consists of all gradient fields with $L^2$-divergence---will be of great importance:
\begin{align}
    G_0^1(\Omega):=\{\nabla\varphi\st \varphi\in H^1_0(\Omega),\;\Delta\varphi\in L^2(\Omega)\}.
\end{align}
For details on this space, such as Helmholtz-decompositions, we refer to \cite{auchmuty2006divergence}.
Finally, we also introduce the space of vector valued Radon measures $\calM(\Omega,\R^n)$, equipped with the total variation norm $\norm{\mu}_{\calM(\Omega,\R^n)}:=|\mu|(\Omega)$, and the closed subspace
\begin{align}
    \calN(\div;\Omega):=\{r\in\calM(\Omega,\R^n)\st\div r=0\}
\end{align}
of solenoidal measures. 
The divergence is understood in the distributional sense, meaning that 
\begin{align}
    \int_\Omega\nabla\varphi\cdot\d r=0,\quad\forall r\in\calN(\div;\Omega),\;\varphi\in C^\infty_c(\Omega).
\end{align}
In order to characterize the subdifferential of $\calJ$, it is useful to express the functional by duality as
\begin{align}\label{eq:fctl}
\calJ(u)=\sup\left\lbrace\int_\Omega-(\div q)\,u\dx\st q\in C^\infty(\overline{\Omega},\R^n),\,\norm{q}_1\leq 1\right\rbrace.
\end{align}

Using this representation we obtain an integral characterization of the subdifferential $\partial\calJ$ as divergences of sums of regular functions and divergence-free measures.
The proof is similar to the characterization of the subdifferential of the total variation in \cite{bredies2016pointwise} and can be found in the appendix.
\begin{prop}[Integral characterization of the subdifferential]\label{prop:characterization_integral}
For $u\in L^2(\Omega)$ it holds
\begin{align}\label{eq:subdifferential_fctl}
\partial\calJ(u)=\left\lbrace-\div q\st q=g+r,\,g\in G^1_0(\Omega),\,r\in\calN(\div;\Omega),\,\int_\Omega-(\div q)\,u\dx=\calJ(u),\,\abs{q}(\Omega)\leq 1\right\rbrace.
\end{align} 
\end{prop}

\begin{definition}[Calibrations]
Any measure $q\in\calM(\Omega,\R^n)$ such that $-\div q\in\partial\calJ(u)$ is called calibration of $u$.
\end{definition}
\begin{rem}[One space dimension]\label{rem:1D_calibrations}
If $\Omega\subset\R$ is an open interval then $\calN(\div;\Omega)$ coincides with constant functions. 
Hence, in this case calibrations $q$ such that $-\div q=-q'\in\partial\calJ(u)$ are always $H(\div)$-functions since the measure part is just a constant. 
\end{rem}
Having the integral characterization from Proposition~\ref{prop:characterization_integral} at hand, we are now interested in explicit forms of calibrations $q$ such that $-\div q\in\partial\calJ(u)$. 
In the following we fix $0\neq u\in W^{1,\infty}_0(\Omega)$ and use the short-cut notation
\begin{align}
    L:=\calJ(u)<\infty.
\end{align}
Furthermore, we define the subset of $\Omega$ where $\nabla u$ attains its maximal modulus as
\begin{align}\label{eq:Omega_max}
\Omega_{\max}:=\left\lbrace x\in\Omega\st |\nabla u(x)|=L\right\rbrace,
\end{align}
a set being defined up to a Lebesgue null-set. 
If we assume for a moment that the calibration $q$ is in $H(\div;\Omega)$, then integrating by parts in \eqref{eq:subdifferential_fctl} according to Proposition~\ref{prop:integration_by_parts} yields
\begin{align}\label{eq:pair_q_gradu}
\calJ(u)=\int_\Omega q\cdot\nabla u\dx,
\end{align}
which suggests that a possible calibration is given by
\begin{align}\label{eq:false_subgradient}
    q(x):=\begin{cases}
\frac{\nabla u(x)}{L}\frac{1}{|\Omega_{\max}|},\quad&x\in\Omega_{\max},\\
0,\quad&\text{else.}
\end{cases}
\end{align}
However, is is obvious from such a choice of $q$ that $\div q\notin L^2(\Omega)$, in general. As already mentioned, an alternative attempt to characterize the subdifferential of $\calJ$ could be to send $p$ to infinity in \eqref{eq:p-laplacians}. However, it is straightforward to see that one formally gets
$$\calJ_p(u)^{1-p}|\nabla u|^{p-2}\nabla u\to q,\quad p\to\infty,$$
where $q$ is again given by \eqref{eq:false_subgradient}. Hence, also this approach fails to describe the subdifferential of $\calJ$. 
Another difficulty comes through the set $\Omega_{\max}$, given by \eqref{eq:Omega_max}, which cannot be expected to have any regularity, as the following example shows.

\begin{example}[Structure of $\Omega_{\max}$]\label{ex:cantor}
In this example we would like to highlight that the structure of the set $\Omega_{\max}$ defined in \eqref{eq:Omega_max} can be highly degenerate. To this end let $\Omega=(0,1)$ and $F\subset\Omega$ be the middle-fourth fat Smith-Volterra-Cantor set which is a closed set with empty interior and positive measure $|F|=1/2$. Furthermore, we set $u(x)=\dist(x,F)$. Then it is straightforward that $\Omega_{\max}=\Omega\setminus F$ is an open set and $\overline{\Omega}_{\max}=\Omega$. In particular, the topological boundary $\partial\Omega_{\max}$ coincides with $F$ and has positive Lebesgue measure. Nevertheless, $u$ has non-empty subdifferential, as we will see.
\end{example}
From \eqref{eq:pair_q_gradu} we can derive yet another regular calibration, given by
\begin{align}\label{eq:regular_calibrations}
    q(x)=f(x)\nabla u(x),
\end{align}
where $f(x)\geq 0$, $\supp(f)\subset\overline{\Omega}_{\max}$ and $\norm{f}_1=1/L$.
Expanding $\div q$ yields
\begin{align}
    \div q=\nabla f\cdot\nabla u+f\Delta u,
\end{align}
where $\Delta u$ denotes the distributional Laplacian of $u$.
Hence in order to satisfy $\div q\in L^2(\Omega)$, function $f$ has to be $H^1(\Omega)$ and meet $f=0$ where $\Delta u$ is singular.
The following examples illustrate that this can be achieved very frequently.
\begin{example}[Measure Laplacians]\label{ex:measure-laplacian}
Let us assume that $u\in W^{1,\infty}_0(\Omega)$ is such that $\Delta u$ is represented by a finite Radon measure. 
In this case it holds that $|\Delta u|\ll\H^{n-1}$ according to \cite[Lem.~2.25]{chen2009gauss}.
Since $f\in H^1(\Omega)$ can be defined in the sense of traces on $n-1$-dimensional sets, one can find a calibration of the form $q=f\nabla u$ where $f$ vanishes on the support of $\Delta u$.
\end{example}
\begin{example}[$\Omega_{\max}$ with non-empty interior]\label{ex:non-empty-int}
Let $u\in W^{1,\infty}_0(\Omega)$ such that $\Omega_{\max}$ has non-empty interior. 
Then one can easily find a smooth non-negative function $f$ supported on some subset of $\Omega_{\max}$ with integral $1/L$. 
In particular, $q=f\nabla u$ will be a calibration.
\end{example}
An important property of calibrations of the form \eqref{eq:regular_calibrations} with a suitable function $f$ is that $q$ is not a measure but a $H(\div)$-function in this case.
In fact, being such a regular of calibrations is equivalent to having the form \eqref{eq:regular_calibrations} as the following proposition shows.
\begin{prop}[Pointwise characterization of regular calibrations]\label{prop:characterization}
Let $0\neq u\in\dom(\calJ)$ and $q\in H(\div;\Omega)$ with $\norm{q}_1=1$. It holds that $-\div q\in\partial\calJ(u)$ if and only if $q=0$ almost everywhere in $\Omega\setminus{\Omega}_{\max}$, and $q\cdot\nabla u=|q||\nabla u|$ almost everywhere in $\Omega$.
\end{prop}
\begin{proof}
Let us show first that $-\div q\in\partial\calJ(u)$ for $q$ as above. Again we use the notation $\calJ(u)=L$. Using the assumptions we compute
\begin{align*}
L&\geq\int_\Omega q\cdot\nabla u\dx=\int_{\Omega}|q||\nabla u|\dx=\int_{{\Omega}_{\max}}|q||\nabla u|\dx\\
&=L\int_{{\Omega}_{\max}}|q|\dx=L.
\end{align*}
Hence, equality holds and we infer
$$\int_\Omega-\div q\,u\dx=\int_\Omega q\cdot\nabla u\dx=L,$$
which shows $-\div q\in\partial\calJ(u)$ according to \eqref{eq:subdifferential_fctl}.

Conversely, let us assume that we have $-\div q\in\partial\calJ(u)$. First, we show that $q=0$ holds a.e. in $\Omega\setminus\Omega_{\max}$. For any $\varepsilon>0$ we define the measurable set
$$\Omega_\varepsilon:=\{x\in\Omega\st|\nabla u(x)|\leq L-\varepsilon\}$$ 
and compute using \eqref{eq:pair_q_gradu}:
\begin{align*}
L&=\calJ(u)=\int_\Omega q\cdot\nabla u\dx=\int_{\Omega_\varepsilon}q\cdot\nabla u\dx+\int_{\Omega\setminus\Omega_\varepsilon}q\cdot\nabla u\dx\\
&\leq (L-\varepsilon)\int_{\Omega_\varepsilon}|q|\dx+L\int_{\Omega\setminus \Omega_\varepsilon}|q|\dx\\
&=L-\varepsilon\int_{\Omega_\varepsilon}|q|\dx.
\end{align*} 
This inequality implies that $q=0$ a.e. on $\Omega_\varepsilon$ and letting $\varepsilon\searrow 0$ we obtain from the continuity of the Lebesgue measure on nested sets that $q=0$ a.e. on $\Omega\setminus\Omega_{\max}$. 

Now we show that $q$ is parallel to $\nabla u$. To this end we re-define the set
$$\Omega_\varepsilon:=\left\lbrace x\in\Omega\st q(x)\cdot\nabla u(x)\leq(1-\varepsilon)|q(x)||\nabla u(x)|,\;|q(x)||\nabla u(x)|\geq\varepsilon\right\rbrace$$ 
for $\varepsilon>0$ and obtain with a similar computation as above that 
$$L\leq L-\varepsilon\int_{\Omega_\varepsilon}|q||\nabla u|\dx,$$
which implies
$$0=\int_{\Omega_\varepsilon}|q||\nabla u|\dx\geq|\Omega_\varepsilon|\eps.$$
This is only possible if $|\Omega_\varepsilon|=0$ and since the sets $\Omega_\varepsilon$ are also nested we again infer from the continuity of the Lebesgue measure that
\begin{align*}
   0&=\left|\bigcup_{\varepsilon>0}\Omega_\varepsilon\right|=\left|\left\lbrace x\in\Omega\st q(x)\cdot\nabla u(x)<|q(x)||\nabla u(x)|,\;|q(x)||\nabla u(x)|>0\right\rbrace\right|\\
   &=\left|\Omega\setminus\left\lbrace x\in\Omega\st q(x)\cdot\nabla u(x)=|q(x)||\nabla u(x)|\right\rbrace\right|,
\end{align*}
which shows that $q$ and $\nabla u$ are parallel a.e. in $\Omega$.
\end{proof}

\subsection{Eigenfunctions}
\label{sec:eigenfunctions}

In this section we would like to study  geometrical properties of eigenfunctions associated to functional $\calJ$, meaning functions $u\in W^{1,\infty}_0(\Omega)$ that meet
\begin{align}
    \lambda u\in\partial\calJ(u),
\end{align}
for some $\lambda>0$. In particular, we study their nodal set 
\begin{align}
N(u)=\{x\in\Omega\st u(x)=0\}
\end{align}
and the set $\Omega_{\max}$ as defined in \eqref{eq:Omega_max}. 
To this end, for the first two statements we assume the regularity condition that the eigenfunctions $u$ under consideration possess a $H(\div)$-calibration~$q$, i.e.
\begin{align}\label{eq:regular_ef_prob}
    \lambda u=-\div q,\quad q\in H(\div;\Omega),\;\norm{q}_1=1,
\end{align}
which makes Proposition~\ref{prop:characterization} applicable.
Remember that the existence of $H(\div)$-calibrations is ensured in many cases (cf.~Remark~\ref{rem:1D_calibrations}, Examples~\ref{ex:measure-laplacian},~\ref{ex:non-empty-int}).
Note that the nodal set $N(u)$ is closed due to continuity of~$u$. 
There are only a few results in the literature which deal with nodal sets of $p$-Laplacian-type eigenfunctions for $p\neq 2$.
In particular, it is not even known whether they have non-empty interior. 
Even if one assumes them to have empty interior, one can only prove lower bounds for their Hausdorff measure, meaning that nodal sets can in principle be very irregular, see \cite{weihnotes,kawohl2017geometry}. 
For the infinity-Laplacian there do not seem to be any results on the geometry of nodal sets. 
Also in our slightly different scenario \eqref{eq:regular_ef_prob}, where the operator is $\partial\calJ$, we cannot fully answer the question. 
However, we can show that $N(u)$ has zero Lebesgue measure if the eigenfunction is sufficiently regular. 
Furthermore, we prove that the interior of the nodal set coincides with the complement of $\overline{\Omega}_{\max}$, which informally means that at each point an eigenfunction is either zero or it has maximal gradient.

\begin{prop}\label{prop:efs_support}
Let $u$ meet \eqref{eq:regular_ef_prob}. Then it holds that 
\begin{align}\label{eq:interior_nodal_set}
    \Omega\setminus\overline{\Omega}_{\max}=\inn(N(u)).
\end{align}
Furthermore, the set $S:=\{x\in\Omega_{\max}\st q(x)=0\}$ has empty interior. 
\end{prop}
\begin{proof}
To avoid trivialities we assume $u\neq 0$ which means $\lambda>0$. 
We use the abbreviation $\Omega_0:=\Omega\setminus\overline{\Omega}_{\max}$. 
Since $\Omega_0$ is open, for any $x_0\in\Omega_{0}$ there is $r>0$ small enough such that $B_r(x_0)\subset\Omega_{0}$. Hence, it holds
\begin{align*}
    \lambda \int_{B_r(x_0)}u^2\dx=-\int_{B_r(x_0)}u\,\div q\dx=\int_{B_r(x_0)}q\cdot\nabla u\dx-\int_{\partial B_r(x_0)}u\,q\cdot\nu\d\H^{n-1}(x)=0,
\end{align*}
since $q=0$ a.e. in $\Omega\setminus\Omega_{\max}\supset\Omega_0$ according to Proposition~\ref{prop:characterization}. This implies $u=0$ on $B_r(x_0)$ and hence $B_r(x_0)\subset \inn(N(u))$. Since $x_0$ was arbitrary we obtain $\Omega_0\subset\inn(N(u))$. For the converse inclusion we take $x_0\in\inn(N(u))$ and $r>0$ such that $B_r(x_0)\subset\inn(N(u))$. Then it holds $u=0$ and $\nabla u=0$ on $B_r(x_0)$, which implies $\inn(N(u))\subset\inn(\Omega\setminus\Omega_{\max})=\Omega\setminus\overline{\Omega}_{\max}=\Omega_0$.

For the second claim, we assume that there is $x_0\in\Omega_{\max}$ and $r>0$ such that $B_r(x_0)\subset S$. Then $u$ cannot be constant on $B_r(x_0)$ since otherwise $|\nabla u|=0$ would hold on $B_r(x_0)$ which contradicts being a subset of $S$. Hence, using that $\int_{B_r(x_0)}u(x)^2\dx>0$ and doing precisely the same computation as above, we obtain a contradiction.
\end{proof}

Using this statement we can easily assert that the set $\Omega_{\max}$ has non-empty interior and hence cannot be too degenerate.

\begin{cor}
Let $u$ meet \eqref{eq:regular_ef_prob}. Then $\Omega_{\max}$ has non-empty interior.
\end{cor}
\begin{proof}
From Proposition~\ref{prop:efs_support} we know that $u=0$ on $\Omega\setminus\overline{\Omega}_{\max}$. If we assume that $\Omega_{\max}$ has empty interior, this implies that $\overline{\Omega}_{\max}=\Omega_{\max}$ and hence $u=0$ on $\Omega\setminus\Omega_{\max}$. Now $u$ is a continuous function which implies that $u=0$ on $\overline{\Omega\setminus\Omega_{\max}}=\Omega$, which is a contradiction. 
\end{proof}

\begin{prop}[Nodal set of eigenfunctions with regularity]
Let $u$ meet \eqref{eq:regular_ef_prob} and assume that $\{u\neq 0\}$ has a Lipschitz boundary. Then it holds $|N(u)|=0$.
\end{prop}
\begin{proof}
If the nodal set has empty interior it holds $N(u)=\partial\{u\neq 0\}$ which means that $|N(u)|=0$ since it coincides with a Lipschitz boundary. Hence we just have to deal with the case that $N(u)$ has non-empty interior. We write $\lambda u=-\div q$ with some calibration $q\in H(\div;\Omega)$. Without loss of generality, let us fix a point $x_0$ in $\partial\{u>0\}\cap N(u)$ and for $\eps>0$ we consider $B^+_\eps(x_0)=B_\eps(x_0)\cap\{u>0\}$. We choose $x_0$ and $\eps>0$ such that $B_\eps(x_0)\cap\{u<0\}=\emptyset$. This is possible due to the continuity of $u$. From the characterization of the subdifferential Proposition~\ref{prop:characterization} we know that $q=0$ a.e. in $N(u)$ and since $N(u)$ has non-empty interior, $q$ has vanishing normal trace on $\partial\{u>0\}\cap B_\eps(x_0)$. This implies
\begin{align*}
    0<\int_{B^+_\eps(x_0)}\lambda u\dx=-\int_{B_\eps^+(x_0)}\div q\dx
    =-\int_{\partial B_\eps(x_0)\cap\{u>0\}} q\cdot\nu \dx.
\end{align*}
Now since $q$ is parallel to $\nabla u$ for small enough $\eps>0$ it holds that $q\cdot\nu\geq 0$ which is a contradiction. Hence, $N(u)$ has zero Lebesgue measure.
\end{proof}

Next we show that every non-negative eigenfunction coincides with a ground state, i.e., is a multiple of the distance function to $\partial\Omega$.
Note that this result \emph{does not} require the regularity condition \eqref{eq:regular_ef_prob} but follows from a simple comparison argument.
\begin{prop}[Uniqueness of non-negative eigenfunction]\label{prop:positivity_efs}
Any non-negative eigenfunction $u\neq 0$ of $\partial\calJ$, meeting $\lambda u\in\partial\calJ(u)$, is a ground state. 
\end{prop}
\begin{proof}
Let us assume that we have a non-negative eigenfunction $u\neq 0$ on $\Omega$ which is no ground state. We can normalize in such a way that $\calJ(u)=1$. Furthermore, we let $d$ denote the distance function which is the unique ground state with $\calJ(d)=1$ according to Theorem~\ref{thm:ground_states_distance}. Then from \cite{zagatti2014maximal} we know that $u\leq d$ holds pointwise almost everywhere in $\Omega$. 
Similar as before we define the set 
$$\Omega_\varepsilon:=\left\lbrace x\in\Omega\st d(x)>u(x)+\varepsilon,\; u(x)>\varepsilon\right\rbrace.$$
Since $u$ is an eigenfunction it holds $\lambda\langle u,v\rangle\leq\calJ(v)$ for all $v\in L^2(\Omega)$, where $\lambda=1/\norm{u}_2^2$. Testing this with $v=d$, using the definition of $\Omega_\varepsilon$ and the fact that $d\geq u$, we obtain
\begin{align*}
    \norm{u}_2^2&\geq\langle u,d\rangle\geq\int_{\Omega_\varepsilon}u(x)(u(x)+\varepsilon)\dx + \int_{\Omega\setminus \Omega_{\varepsilon}}u(x)d(x)\dx \\
    &\geq\int_\Omega u(x)^2\dx +\varepsilon\int_{\Omega_\varepsilon}u(x)\dx\\
    &\geq\norm{u}_2^2+\varepsilon^2|\Omega_\varepsilon|,
\end{align*}
which tells us that $|\Omega_\varepsilon|=0$. Letting $\varepsilon$ tend to zero we infer as before that almost everywhere in $\Omega$ it holds $u=d$ or $u=0$. Since, however both $u$ and $d$ are continuous functions and by assumption $u\neq 0$, we find that $u=d$ holds almost everywhere in $\Omega$. 
\end{proof}
Using this uniqueness of non-negative eigenfunctions together with the results in \cite{bungert2019asymptotic} we obtain the result that the gradient flow of $\calJ$ asymptotically converges to the distance function.
\begin{thm}[Asymptotic profiles]\label{thm:profiles_are_distance_fcts}
Let $u(t)$ be the solution of the gradient flow \eqref{gradflow} with respect to $\calJ$ and datum $f\geq 0$. Denote the finite extinction time of the flow by $T$. Then $u(t)/\norm{u(t)}_2$ converges strongly in $L^2(\Omega)$ to a multiple of the distance function as $t\nearrow T$.
\end{thm}
\begin{proof}
Since $\dom(\calJ)=W^{1,\infty}_0(\Omega)$ is compactly embedded in $L^2(\Omega)$ we infer from \cite[Thm.~2.5]{bungert2019asymptotic} that $u(t)/\norm{u(t)}_2$ has a subsequence which strongly converges to an eigenfunction. Now \cite[Thm.~2.6]{bungert2019asymptotic} implies that the whole sequence converges to a non-negative eigenfunction. From Proposition~\ref{prop:positivity_efs} and Theorem~\ref{thm:ground_states_distance} we conclude that this eigenfunction has to be a multiple of the distance function.
\end{proof}

\begin{example}[Distance function of the $n$-sphere]
In this example we study the distance function $d$ of the $n-1$-sphere $S_{n-1}:=\{x\in\R^n\st|x|=1\}$, where we choose $\Omega=B_1(0)$. We already know from Theorem~\ref{thm:ground_states_distance} that the distance function is an eigenfunction, i.e., $\lambda d=-\div q$ where $\lambda=\calJ(d)/\norm{d}_2^2=1/\norm{d}_2^2$ and $\norm{q}_1\leq 1$.
Furthermore, since $q$ is parallel to $\nabla u$, we can write $q$ as $q=f\nabla u$ with $f\geq 0$. 
In the following we would like to detail function~$f$. 
We claim that in spherical coordinates it holds
$$f(r)=\lambda\left(\frac{r}{n}-\frac{r^2}{n+1}\right).$$
The radial component of the gradient of $d(r)=1-r$ is given by $\nabla_r d=d'(r)=-1$ and there is no angular component. Hence, we obtain that the radial component of the calibration vector field $q=f\nabla d$ is given by $q_r(r)=\lambda_n\left(\frac{r^2}{n+1}-\frac{r}{n}\right)$ which implies
\begin{align*}
    -\div(f(r)\nabla d(r))&=-\frac{1}{r^{n-1}}\frac{\d}{\d r} (r^{n-1}q_r(r))\\
    &=\lambda\frac{1}{r^{n-1}}\frac{\d}{\d r}\left(\frac{r^n}{n}-\frac{r^{n+1}}{n+1}\right)\\
    &=\lambda(1-r)\\
    &=\lambda d(r).
\end{align*}
Furthermore, it is straightforward to check that $\norm{q}_1=1$.
Note that the qualitative behavior of $f$ changes with the dimension $n\in\N$. In particular, $f(r)$ attains its maximum for $r=\frac{n+1}{2n}$ which tends to $1/2$ as the dimension grows. Furthermore, $f$ has roots at $r=0$ and $r=\frac{n+1}{n}$ which tends to one from above. Furthermore, the value of $f(1)$ diverges. 
\end{example}

\begin{example}[A basis of 1D-eigenfunctions]\label{ex:1d_efs}
In this example we construct a set of 1D-eigenfunctions on the interval $\Omega=[-1,1]$ which constitutes a Riesz basis of $L^2(\Omega)$. They disintegrate into odd and even ones with respect to the center of the interval and can be constructed by simple gluing principles. We start with the odd ones which we denote by $(v_n)_{n\in\N}$. Let $\Omega=\bigcup_{k=1}^{2n}\Omega_k$ a decomposition of $\Omega$ into $2n$ intervals of length $1/n$ such that $\Omega_k\leq \Omega_{k+1}$ holds for all $k=1,\dots,2n-1$. Letting $d_k$ denote the distance function of $\Omega_k$ we set
$$
u_{n}\vert_{\Omega_k}(x)=(-1)^{k+1}d_k(x).
$$ 
Note that all functions $u_n$ satisfy $u_n(0)=0$ and $u(-x)=-u(x)$. Furthermore, it is worth noting that the functions $(u_n)$ form a orthogonal set. This follows directly from the fact that $u_n$ consists of equally many positive and negative distance functions. The eigenvalues of $u_n$ can be easily computed and are given by
$$R(u_n)=\frac{1}{\norm{u_n}_2}=\sqrt{\frac{3}{2}}2n.$$

The even eigenfunctions $(v_n)$ are generated similarly. Here we divide the interval $\Omega$ into $2n-1$ intervals $\Omega_k$ of length $2/(2n-1)$ such that $\Omega=\bigcup_{k=1}^{2n-1}\Omega_k$ and $\Omega_k\leq \Omega_{k+1}$ holds for all $k=1,\dots,2n-2$. Letting $d_k$ again denote the distance function of $\Omega_k$ we set
$$
v_{n}\vert_{\Omega_k}(x)=(-1)^{k+1}d_k(x).
$$
All functions $v_n$ satisfy $v_n(-x)=v_n(x)$ and, in particular, $v_1$ coincides with the distance function of $\Omega$ which is even and a ground state. Note that functions $(v_n)$ are \emph{not} mutually orthogonal. Their eigenvalues are given by
$$R(v_n)=\frac{1}{\norm{v_n}_2}=\sqrt{\frac{3}{2}}(2n-1).$$

Figure~\ref{fig:1d_efs} shows the first four eigenfunctions $\{v_1,u_1,v_2,u_2\}$ sorted by eigenvalue. Note that--up to the factor $\sqrt{3/2}$---the eigenvalues of $u_n$ and $v_n$ precisely count the numbers of peaks or oscillations.

The fact that $\{u_n,v_n\st n\in\N\}$ is a Riesz basis of $L^2(\Omega)$ was proven in \cite{binding2006basis}.
\end{example}

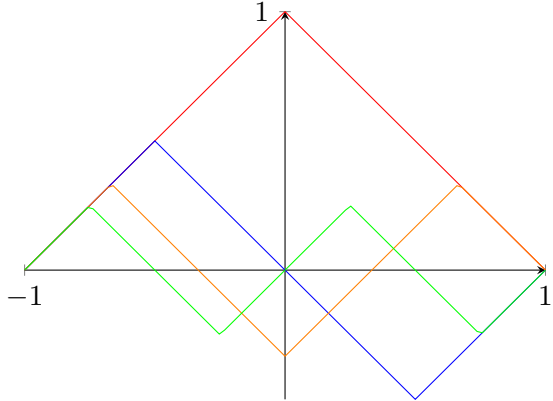
\begin{figure}[h!]
    \centering
    \begin{tikzpicture}
        \begin{axis}[
    axis lines = center,
    axis equal image,
    xtick = {-3,0,3},
    xticklabels = {$-1$,$0$,$1$},
    ytick = {0,3},
    yticklabels = {$0$,$1$}
]

\def\first(#1){and(#1 >= -3, #1 < 0) * (#1+3) + 
    and(#1 >= 0, #1 < 3) * (3-(#1))
    }
\addplot [
    domain=-3:3, 
    samples=101, 
    color=red,
]
{\first(x)};

\def\second(#1){and(#1 >= -3, #1 < -1.5) * (#1+3) + 
    and(#1 >= -1.5, #1 < 0) * (3-(#1+3)) +
    and(#1 >= 0, #1 < 1.5) * (3-(#1+3)) +
    and(#1 >= 1.5, #1 < 3) * ((#1+3)-6) 
    }
\addplot [
    domain=-3:3, 
    samples=101, 
    color=blue,
]
{\second(x)};

\def\third(#1){and(#1 >= -4, #1 < -2) * (#1+3) + 
    and(#1 >= -2, #1 < -1) * (2-(#1+3)) +
    and(#1 >= -1, #1 < 0) * (2-(#1+3)) +
    and(#1 >= 0, #1 < 1) * ((#1+3)-4) +
    and(#1 >= 1, #1 < 2) * ((#1+3)-4) +
    and(#1 >= 2, #1 < 3) * (6-(#1+3))
    }
\addplot [
    domain=-3:3, 
    samples=101, 
    color=orange,
]
{\third(x)};

\def\fourth(#1){and(#1 >= -3, #1 < -2.25) * (#1+3) + 
    and(#1 >= -2.25, #1 < -1.5) * (1.5-(#1+3)) +
    and(#1 >= -1.5, #1 < -.75) * (1.5-(#1+3)) +
    and(#1 >= -.75, #1 < 0) * ((#1+3)-3) +
    and(#1 >= 0, #1 < 0.75) * ((#1+3)-3) +
    and(#1 >= 0.75, #1 < 1.5) * (4.5-(#1+3)) +
    and(#1 >= 1.5, #1 < 2.25) * (4.5-(#1+3)) + 
    and(#1 >= 2.25, #1 < 3) * ((#1+3)-6)
    }
\addplot [
    domain=-3:3, 
    samples=100, 
    color=green,
]
{\fourth(x)};

 \end{axis}
     \end{tikzpicture}
     \caption{First four eigenfunctions with increasing number of oscillations}
     \label{fig:1d_efs}
\end{figure}

\section{Explicit solution of gradient flow and variational problem}
\label{sec:explicit_sol}
We already know from Theorem~\ref{thm:profiles_are_distance_fcts} that the solution of the gradient flow \eqref{gradflow} with respect to $\calJ$ asymptotically behaves like the distance function of the domain. In the following, we prove that for sufficiently regular domains and constant initialization, one can compute the solution of the gradient flow analytically. In addition, this solution also solves the variational regularization problem \eqref{varprob} associated to $\calJ$. Notably, this solution exhibits an interesting behavior of its level sets which reminds of the fast marching algorithm or other level set approaches (cf.~\cite{sethian1996fast,sussman1994level}). Before we construct these analytic solutions we start with some definitions regarding the kind of domains we consider.
 
\begin{definition}[Inner parallel body]
Let $\Omega\subset\R^n$ be an open set and let $d(x):=\dist(x,\partial\Omega)$ denote the distance function to $\partial\Omega$. Then
\begin{align}
    \Omega_\tau:=\{x\in\Omega\st d(x)\geq \tau\}
\end{align}
is called the inner parallel body of $\Omega$ with distance $\tau>0$. 
\end{definition}

\begin{definition}[Perimeter bound for inner parallel body]\label{def:perimeter_bound}
We say that $\Omega$ admits a perimeter bound for its inner parallel bodies if there is $\tilde{r}>0$ and $0<\tilde{\tau}\leq \tilde{r}$ such that
\begin{align}\label{ineq:perimeter_bound}
P(\Omega_\tau)\geq P(\Omega)\left(1-\frac{\tau}{\tilde{r}}\right)^{n-1},\quad\forall 0\leq\tau\leq\tilde{\tau}.
\end{align} 
\end{definition}

\begin{example}[Convex domains]\label{ex:perimeter_bound}
According to \cite{larson2016bound} convex domains $\Omega\subset\R^n$ always fulfill a perimeter bound like \eqref{ineq:perimeter_bound} with $\tilde{r}=\tilde{\tau}=r$ where $r=\max_{x\in\Omega}\dist(x,\partial\Omega)$ denotes the in-radius of $\Omega$. Furthermore, if $\Omega$ is homothetic to its \emph{form body} then \eqref{ineq:perimeter_bound} becomes an equality. This is the case, for instance, if $\Omega$ is a ball or a polytope whose faces are tangential to the largest ball which can be inscribed in $\Omega$.
\end{example}

\begin{example}[L-shaped domain]
Let us consider an L-shaped domain with equal width and height given by $L>0$ and thickness $\delta\in(0,L)$. For instance, one could set $\Omega:=[0,L]^2\setminus [0,L-\delta]^2\subset\R^2$. We are interested in whether $\Omega$ admits the perimeter bound \eqref{ineq:perimeter_bound}. To this end we notice that the perimeter of $\Omega$ is given by $P(\Omega)=4L$ and the perimeter of~$\Omega_\tau$ for $0\leq\tau\leq\min(L-\delta,\delta/2)$ can be computed as
\begin{align*}
P(\Omega_\tau)&=2(L-2\tau)+2(\delta-2\tau)+2(L-\delta-\tau)+\frac{1}{4}2\tau\pi\\
&=4L\left(1-\tau\frac{20-\pi}{8L}\right)\\
&=P(\Omega)\left(1-\frac{\tau}{\tilde{r}}\right),
\end{align*}
where $\tilde{r}=8L/(20-\pi)$. 
The number $\tilde{\tau}$ is given by $\tilde{\tau}=\min(L-\delta,\delta/2)$ and meets $\tilde{\tau}<\tilde{r}$.
Hence, the L-shape admits the perimeter bound~\eqref{ineq:perimeter_bound}.
\end{example}

Before we turn to the main theorem of this section, which constructs the explicit solution, we have to study the properties of a geometric integral which will appear in the proof. 
\begin{lemma}\label{lem:geometric_integral}
Let $\Omega\subset\R^n$ be a domain, $d(x):=\dist(x,\partial\Omega)$ denote the distance function to $\partial\Omega$, and $r:=\max_{x\in\Omega}d(x)$ the in-radius of $\Omega$.
Then for $k\in\N$, we define the function
\begin{align}\label{eq:function_I}
    I_k(g):=\int_{\Omega\setminus\Omega_{rg}}d(x)^k\dx,\quad 0\leq g\leq 1.
\end{align}
\begin{itemize}
    \item For all $k\in\N$ it holds that $I_k(0)=0$, $I_k$ is monotonously increasing and differentiable with 
    \begin{align}\label{eq:derivative_I}
         I_k'(g) &= P(\Omega_{rg})r^{k+1}g^k,\quad \forall 0<g<1.
    \end{align}
    \item If $\Omega$ admits the perimeter bound~\eqref{ineq:perimeter_bound} for its inner parallel body, then function $I_2$ admits the following estimate for all $0\leq g \leq \frac{\tilde{\tau}}{r}$
    \begin{align}
        I_2(g) &\geq \frac{\tilde{r}^3P(\Omega)}{n}\Bigg\lbrace\frac{2}{(n+1)(n+2)}\left[1-\left(1-\frac{rg}{\tilde{r}}\right)^{n+2}\right]-\frac{2}{n+1}\left(1-\frac{rg}{\tilde{r}}\right)^{n+1} \frac{rg}{\tilde{r}} \notag\\
        &\qquad\qquad\qquad-\left(\frac{rg}{\tilde{r}} \right)^2\left(1-\frac{rg}{\tilde{r}} \right)^n\Bigg\rbrace. \label{ineq:lower_bound_I}
    \end{align}
\end{itemize}
\end{lemma}
\begin{proof}
It is trivial that $I_k(0)=0$ and $I_k$ is monotonously increasing. 
For showing~\eqref{eq:derivative_I} we let $\tilde{g}<g$ and compute using the coarea formula
\begin{align*}
    I_k(g)-I_k(\tilde{g})=\int_{S_{r\tilde{g},rg}}d(x)^k\dx=\int_{r\tilde{g}}^{rg}P(\Omega_t)t^k\d t.
\end{align*}
Consequently, we obtain
\begin{align*}
    I_k'(g)=\lim_{\tilde{g}\to g}\frac{I_k({g})-I_k(\tilde{g})}{g-\tilde{g}}=r\lim_{\tilde{g}\to g}\frac{1}{rg-r\tilde{g}}\int_{r\tilde{g}}^{rg} P(\Omega_t) t^k\d t = r P(\Omega_{rg}) (rg)^k=P(\Omega_{rg})r^{k+1}g^k.
\end{align*}
To evaluate $I_2(g)$ we make use of the layer cake formula, which states that the integral of a non-negative function $h:\Omega\to\R$ can be computed as
\begin{align}\label{eq:layer_cake}
\int_\Omega h(x)\dx=\int_0^\infty|\{x\in\Omega\st h(x)>t\}|\d t.
\end{align}
Let us first estimate the Lebesgue measure of the strip $S_{s,t}:=\Omega_s\setminus\Omega_t$ where $s<t$. By using the coarea formula and the perimeter bound \eqref{ineq:perimeter_bound} it holds for $0\leq s\leq t<\tilde{\tau}$
\begin{align}\label{ineq:estimate_strip}
|S_{s,t}|=\int_s^t P(\Omega_\tau)\d\tau\geq P(\Omega)\int_s^t\left(1-\frac{\tau}{\tilde{r}}\right)^{n-1}\d\tau=\frac{\tilde{r}P(\Omega)}{n}\left[(1-s/\tilde{r})^n-(1-t/\tilde{r})^n\right].
\end{align}
Letting $h_g(x):= d(x)^2\chi_{\Omega\setminus\Omega_{rg}}$ for $0\leq g \leq \frac{\tilde{\tau}}{r}$ we infer from \eqref{eq:layer_cake} and \eqref{ineq:estimate_strip}
\begin{align*}
I_2(g)&=\int_\Omega h_g(x)\dx\\ &=\int_0^{(rg)^2}|\{x\in\Omega\st t<h_g(x)<(rg)^2\}|\d t\\
&=\int_0^{(rg)^2}|S_{\sqrt{t},rg}|\d t\\
&\geq \frac{\tilde{r}P(\Omega)}{n}\int_0^{(rg)^2}(1-\sqrt{t}/\tilde{r})^n-(1-rg/\tilde{r})^n \d t\\
&=\frac{\tilde{r}^3P(\Omega)}{n}\Bigg\lbrace\frac{2}{(n+1)(n+2)}\left[1-\left(1-\frac{rg}{\tilde{r}}\right)^{n+2}\right]-\frac{2}{n+1}\left(1-\frac{rg}{\tilde{r}}\right)^{n+1} \frac{rg}{\tilde{r}}\\
&\qquad\qquad\qquad-\left(\frac{rg}{\tilde{r}} \right)^2\left(1-\frac{rg}{\tilde{r}} \right)^n\Bigg\rbrace,
\end{align*}
where we used elementary integration for that last equality.
This shows~\eqref{ineq:lower_bound_I}.
\end{proof}

\begin{thm}\label{thm:explicit_solution_gradflow}
Under the conditions of Lemma~\ref{lem:geometric_integral} there is $t_*>0$ such that the initial value problem
\begin{align}\label{eq:ivp}
\begin{cases}
g'(t)&=\frac{g(t)^2}{I_2(g(t))},\quad t>0,\\
g(0)&=0,
\end{cases}
\end{align}
where $I_2$ is given by~\eqref{eq:function_I} for $k=2$, has a solution for $t\in[0,t_*]$.
Furthermore,
\begin{align}\label{eq:explicit_solution}
u(t,x)=
\begin{cases}
\min\left(\frac{1}{g(t)} d(x),r\right),\quad&0\leq t<t_*,\\
\frac{1}{\norm{d}_2^2}\left(\norm{d}_2^2+t_*-t\right)_+d(x),&t\geq t_*,
\end{cases}
\end{align}
solves the gradient flow \eqref{gradflow} with respect to $\calJ$ and datum $f\equiv r$.
\end{thm}
\begin{proof}
Note that since $d$ is an eigenfunction of $\partial\calJ$, it is known that the dynamics for $t\geq t_*$ will linearly shrink the eigenfunction until extinction (cf.~\cite{bungert2019nonlinear,burger2016spectral}, for instance). 
Hence, we will focus on the initial dynamics and first show that the initial value problem~\eqref{eq:explicit_solution} has a solution $g(t)$, which persists long enough such that $g(t_*)=1$ for some $t_*>0$.
Afterwards, we will show that~\eqref{eq:explicit_solution} solves the gradient flow.
\paragraph{Step 1}
First we study the fine behavior of the lower bound in \eqref{ineq:lower_bound_I} as $g\searrow 0$. 
To this end, one notes that the derivative of the right hand side in \eqref{ineq:lower_bound_I} with respect to $g$ is given by $C(\frac{rg}{\tilde{r}})^2(1-\frac{rg}{\tilde{r}})^{n-1}$ with a positive constant $C=C(n,\Omega)>0$, which by L'H\^{o}pital's rule shows that 
\begin{align*}
    \liminf_{g\searrow 0}\frac{I_2(g)}{g^3}>0.
\end{align*}
In particular, for the ODE $g'(t)=g(t)^2/I_2(g(t))$ this implies that for small times $t>0$ the right hand side is dominated by $1/g(t)$. The fact that the problem
$$\phi'(t)=1/\phi(t),\qquad\phi(0)=0$$
has a solution (namely $\phi(t)=\sqrt{2t}$) implies existence of a solution to \eqref{eq:ivp} for small times. 
Analogously, due to the fact that $I_2(g)$ is bounded from above by the value $I_2(1)$ according to Lemma~\ref{lem:geometric_integral}, the right hand side in~\eqref{eq:ivp} is bounded from below by $g(t)^2/I_2(1)$.
Hence, if we fix $t_0>0$ in the existence interval of $g$, it holds for all $t\geq t_0$ in the existence interval that $g(t)\geq\phi(t-t_0)$, where $\phi$ solves
$$\phi'(t)=\phi(t)^2/I_2(1),\qquad \phi(0)=g(t_0)>0.$$
This problem has the blow-up solution $\phi(t)=g(t_0)I_2(1)/(I_2(1)-g(t_0)t)$ and hence we infer the existence of $t_*>0$ such that $g(t_*)=1$. 
\paragraph{Step 2}
It remains to be shown that~\eqref{eq:explicit_solution} solves the gradient flow. 
Obviously, it holds $u(0,x)=r=f(x)$ for all $x\in\Omega$ since $g(0)=0$.
Furthermore, we can compute that
$$\partial_t u(t,x)=-\frac{1}{2}\frac{g'(t)}{g(t)^2} d(x)\left[1-\sgn(d(x)-rg(t))\right],$$
which yields that for all $0<t<t_*$ we have
\begin{align*}
\langle-\partial_tu(t),u(t)\rangle=\frac{g'(t)}{g(t)^3}\underbrace{\int_{\Omega\setminus\Omega_{rg(t)}}d(x)^2\dx}_{=:I_2(g(t))}=\frac{1}{g(t)}=\calJ(u(t)),
\end{align*}
using that $g$ solves~\eqref{eq:ivp}.
Hence, we have shown $\langle-\partial_tu(t),u(t)\rangle=\calJ(u(t))$ and it remains to be shown that $\langle-\partial_t u(t),v\rangle\leq\calJ(v)$ holds for all $v\in W^{1,\infty}_0(\Omega)$.
We compute using that $g(t)$ solves \eqref{eq:ivp}:
$$\langle-\partial_t u(t),v\rangle=\frac{g'(t)}{g(t)^2}\int_{\Omega\setminus\Omega_{rg(t)}}d(x) v(x)\dx=\frac{1}{I_2(g(t))}\int_{\Omega\setminus\Omega_{rg(t)}}d(x) v(x)\dx.$$
For any $x\in\Omega$ we choose $y=y_x\in\partial\Omega$ such that $|x-y_x|=\min_{y\in\partial\Omega}|x-y|=d(x)$. Then using the Lipschitz continuity of $v$ (cf.~Remark~\ref{rem:lipschitz}) and $v(y_x)=0$, we obtain
$$|v(x)|=|v(x)-v(x_y)|\leq\calJ(v)d(x).$$
Putting things together we can finish the proof by calculating
\begin{align*}
\langle-\partial_tu(t),v\rangle\leq\frac{1}{I_2(g(t))}\int_{\Omega\setminus\Omega_{rg(t)}}d(x)|v(x)|\dx\leq\frac{\calJ(v)}{I_2(g(t))}\int_{\Omega\setminus\Omega_{rg(t)}}d(x)^2\dx=\calJ(v),
\end{align*}
which yields that $-\partial_t u(t)\in\partial\calJ(u(t))$.
\end{proof}

\begin{cor}[Motion of level sets]\label{cor:levelsets}
Under the conditions of Theorem~\ref{thm:explicit_solution_gradflow} the level sets 
\begin{align*}
    \Gamma_c(t)=\{x\in\Omega\st u(x)=c\}
\end{align*}
of $u(t)$ at level $c\geq 0$ and time $0\leq t\leq t_*$ are given by:
\begin{subequations}\label{eq:level_sets}
 \begin{align}
     \Gamma_c(t)&=\{x\in\Omega\st d(x)=c g(t)\},\qquad 0\leq c< r,\\
     \Gamma_r(t)&=\{x\in\Omega\st d(x)\geq rg(t)\}.
 \end{align}
\end{subequations}
This means that the level sets are inner parallel set of $\partial\Omega$ moving with a velocity that is proportional to both the level and function $g'(t)\approx 1/\sqrt{t}$ for small $t$.   
\end{cor}

\begin{rem}[Comparison to level set methods]
A traditional way to compute distance functions was proposed in \cite{sussman1994level} and uses the following PDE
\begin{align}\label{eq:time_dep_eikonal-1}
    \begin{cases}
        u(0,x)=f(x), \quad &x\in\R^n\\
        \partial_t u(t,x)+\sgn(f(x))(|\nabla u(t,x)|-1)=0,\quad &(t,x)\in(0,\infty)\times\R^n,
    \end{cases}
\end{align}
where the initial datum $f$ fulfills $f>0$ in $\Omega$, $f<0$ in $\R^n\setminus\Omega$, and $f=0$ in $\partial\Omega$. The steady state of this equation solves the Eikonal equation $|\nabla u|=1$ and coincides with the signed distance function of $\Omega$. Similarly, in \cite{lee2017revisiting} the authors use the PDE
\begin{align}\label{eq:time_dep_eikonal}
    \partial_t u(t,x)+|\nabla u(t,x)|=0
\end{align}
for a redistancing procedure that converges to the signed distance function as well. It is straightforward to see that points $x(t)$ in the level sets of the solutions of \eqref{eq:time_dep_eikonal-1} move with the following velocity
\begin{align}
    \dot{x}(t)&=\sgn(f(x(t)))\frac{|\nabla u(t,x(t))|-1}{|\nabla u(t,x(t))|}\frac{\nabla u(t,x(t))}{|\nabla u(t,x(t))|}.
\end{align}
In particular, for regions where the gradient is very steep the level sets of \eqref{eq:time_dep_eikonal-1} move with unit velocity whereas the level sets~\eqref{eq:level_sets} of our gradient flow solution move with velocity $\approx1/\sqrt{t}$ for small times.
\end{rem}

\begin{example}[One-dimensional interval]
Let us consider the gradient flow \eqref{gradflow} with datum $f:=1$ on the domain $\Omega:=(-1,1)$. Then the solution is given by
\begin{align}
u(t,x)=\begin{cases}
\min\left(\frac{1}{\sqrt{3t}}(1-|x|),1\right),\qquad&0\leq t<\frac{1}{3},\\
\frac{3}{2}\left(1-t\right)_+(1-|x|),\qquad&t\geq\frac{1}{3}.
\end{cases}
\end{align}
\end{example}

\begin{example}[Two-dimensional disk]
We study the case $\Omega=B_1(0)\subset\R^2$ where $r=1$. From Example~\ref{ex:perimeter_bound} we know that \eqref{ineq:lower_bound_I} is in fact an equality since $\Omega$ is a ball and thus it holds
$$I_2(g)=\frac{\pi}{6}g^3(4-3g).$$
Hence the initial value problem \eqref{eq:ivp} becomes
\begin{align}\label{eq:ivp_for_circle}
g'(t)=\frac{g(t)^2}{I_2(g(t))}=\frac{6}{\pi}\frac{1}{g(t)}\frac{1}{4-3g(t)},\quad g(0)=0.
\end{align}
In Figure~\ref{fig:g_for_circle} we plot a numerical approximation for $g$. In particular, we see that for small times $t>0$ function $g(t)$ is proportional to the square root of $t$ whereas these dynamics change for larger times, as it can be expected from \eqref{eq:ivp_for_circle}. 
\begin{figure}[h!]
\centering
{\includegraphics[width=0.49\textwidth,trim = 4cm 10cm 4cm 10cm, clip]{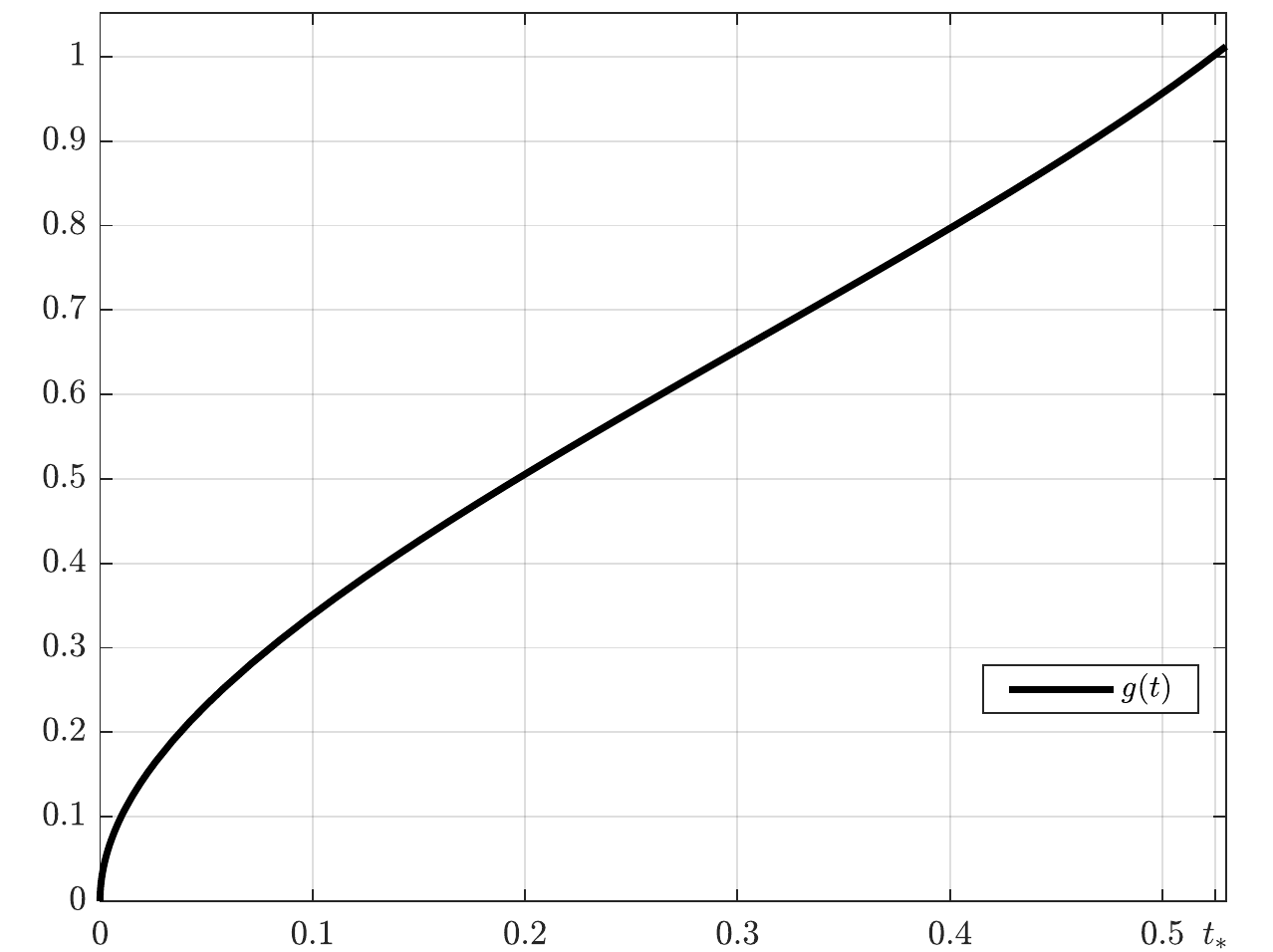}}
\caption{$g(t)$ for the unit circle\label{fig:g_for_circle}}
\end{figure}
\end{example}
Next, we prove that the analytic solution \eqref{eq:explicit_solution} also solves the variational regularization problem \eqref{varprob}. 

\begin{thm}[Variational problem]
Under the conditions of Theorem~\ref{thm:explicit_solution_gradflow} it holds that \eqref{eq:explicit_solution} is the unique solution of
\begin{align}\label{eq:varprob}
    \min_{u\in W^{1,\infty}_0(\Omega)}\frac{1}{2}\norm{u-f}^2_2+t\norm{\nabla u}_\infty,
\end{align}
where $f\equiv r$.
\end{thm}
\begin{proof}
The optimality condition for problem \eqref{eq:varprob} are given by $(f-u(t))/t\in\partial J(u(t))$, which is sufficient for optimality due to convexity of~\eqref{eq:varprob}.
We first show that $(f-u(t))/\tilde{t}\in\partial\calJ(u(t))$ where
\begin{align}
    \tilde{t}:=rI_1(g(t))-\frac{1}{g(t)} I_2(g(t)),
\end{align}
and the functions $I_k$ for $k\in\{1,2\}$ are given by~\eqref{eq:function_I}.
In a second step we show that $\tilde{t}=t$.
\paragraph{Step 1} By the definition of $\tilde{t}$ and the functions $I_k$ it holds
\begin{align*}
    \left\langle\frac{f-u(t)}{\tilde{t}},u(t)\right\rangle&=\frac{1}{\tilde{t}}\int_{\Omega\setminus\Omega_{rg(t)}}\left( r-\frac{d(x)}{g(t)}\right)\frac{d(x)}{g(t)}\dx \\
    &= \frac{1}{\tilde{t}}\left(\frac{r}{g(t)}I_1(g(t))-\frac{1}{g(t)^2}I_2(g(t))\right)\\
    &=\frac{1}{g(t)}=\calJ(u(t)).
\end{align*}
Furthermore, for any $v\in W^{1,\infty}_0(\Omega)$ one computes
\begin{align*}
    \left\langle \frac{f-u(t)}{\tilde{t}},v \right\rangle=\frac{1}{\tilde{t}}\int_{\Omega\setminus\Omega_{rg(t)}}\left(r-\frac{d(x)}{g(t)}\right)v(x)\dx\leq\calJ(v),
\end{align*}
where we used Lipschitz continuity of $v$ just as in the proof of Theorem~\ref{thm:explicit_solution_gradflow}. Hence, we have established $(f-u(t))/\tilde{t}\in\partial\calJ(u(t))$.
\paragraph{Step 2}
To show $\tilde{t}=t$ we use the chain rule and \eqref{eq:derivative_I} from Lemma~\ref{lem:geometric_integral} for $k\in\{1,2\}$ to obtain
\begin{align*}
    \frac{\d}{\d t}\tilde{t}&=rg'(t)I_1'(g(t))+\frac{g'(t)}{g(t)^2}I_2(g(t))-\frac{g'(t)}{g(t)}I_2'(g(t))\\
    &=r g'(t)P(\Omega_{rg(t)})r^2g(t)+\frac{g'(t)}{g(t)^2}I_2(g(t)) -\frac{g'(t)}{g(t)} P(\Omega_{rg(t)})r^3g(t)^2\\
    &=\frac{g'(t)}{g(t)^2}I_2(g(t))=1,
\end{align*}
where the last equality holds since $g(t)$ solves the ODE \eqref{eq:ivp}. 
Furthermore, using L'H\^{o}pital's rule and~\eqref{eq:derivative_I} it holds
\begin{align*}
    \lim_{t\searrow 0}\tilde{t}=\lim_{t\searrow 0}\left[rI_1(g(t))-\frac{1}{g(t)}I_2(g(t))\right] 
    =-\lim_{t\searrow 0}\frac{I_2(g(t))}{g(t)}
    =-\lim_{t\searrow 0}I_2'(g(t)) = 0,
\end{align*}
which finally implies that $\tilde{t}=t$.
\end{proof}

\section{Extreme points}\label{sec:extreme_points}
In this section we aim to characterize extreme points of the unit ball $B_\calJ$ of $\calJ$, which is given by
\begin{align}\label{eq:unit_ball}
    B_\calJ:=\left\lbrace u\in L^2(\Omega)\st \calJ(u)\leq 1\right\rbrace,
\end{align}
and is a convex set and closed set in $L^2(\Omega)$. For a general convex set $C$, its extreme points are defined as
\begin{align}
    \extr(C):=\left\lbrace u\in C\st \nexists\,v\neq w\in C,\,\lambda\in(0,1)\st u=\lambda v+(1-\lambda)w\right\rbrace,
\end{align}
meaning the extreme points of $C$ are precisely those points which cannot be expressed through a non-trivial convex combination of other points in $C$.

The set of extreme points of the unit ball of a similar functional has already been studied in \cite{farmer1994extreme,smarzewski1997extreme}. There the authors considered the Lipschitz semi-norm of functions on a metric space which have a prescribed value in one point. Our situation is more complicated since we prescribe a value on the whole boundary of $\Omega$. 

The following theorem characterizes the extreme points of $B_\calJ$ analogously to the results in \cite{farmer1994extreme}. In a nutshell, a function in $B_\calJ$ is extreme if and only if for almost every point in the domain there exists a path from the point to the boundary of the domain such that the gradient of the function has unit modulus along this path. 
To this end one introduces the quantity
\begin{align}
    \eps_{x,z}^u:=\inf\left\lbrace \eps>0 \st |x_{i-1}-x_i|-\eps_i\leq |u(x_{i-1})-u(x_i)|\right\rbrace,\label{eq:def_eps}
\end{align}
where the infimum is computed over all finite sequences of non-negative numbers $(\eps_i)_{i=1,\dots,n}$ fulfilling $\sum_{i=1}^n\eps_i\leq\eps$, and points $(x_i)_{i=0,\dots,n}$ with $x_0=z,x_1,\dots,x_n=x$. 

Loosely speaking, $\eps_{x,z}^u$ measures the deviation of the gradient norm from being $1$, while moving on a path from $x$ to the boundary point $z$.
The following theorem states that if the infimum of \eqref{eq:def_eps} over all boundary points $z$ is zero, $u$ is an extreme function.
We postpone the proof to the appendix since it is a lengthy generalization of the proof in~\cite{farmer1994extreme}.

\begin{thm}[Characterization of extreme points]\label{thm:extremal_points}
It holds that $u\in\extr(B_\calJ)$ if and only if for almost all $x\in\Omega$ it holds 
\begin{align}
    \inf_{z\in\partial\Omega}\eps_{x,z}^u=0,\label{eq:cons_inf_eps}
\end{align}
where $\eps_{x,z}^u$ is given by \eqref{eq:def_eps}.
\end{thm}

In the following proposition we sandwich the set of extreme points between two other interesting sets, namely those functions whose gradient has modulus one everywhere except from a set with zero measure or non-empty interior, respectively.

\begin{prop}[Sandwiching extreme points]\label{prop:extreme}
It holds that 
\begin{align}
    \left\lbrace u\in B_\calJ\st|\Omega\setminus\Omega_{\max}|=0 \right\rbrace \subset
    \extr(B_\calJ)\subset
    \left\lbrace u\in B_\calJ\st \inn\left(\Omega\setminus\Omega_{\max}\right)=\emptyset\right\rbrace.
\end{align}
\end{prop}
\begin{proof}
For the first inclusion we take $u\in W^{1,\infty}_0(\Omega)$ with $\abs{\nabla u}=1$ almost everywhere, and assume that there are $v\neq w\in B_\calJ$ and $\lambda\in(0,1)$ such that $u=\lambda v+(1-\lambda)w$. Defining the set $\Omega_\eps=\{x\in\Omega\st \abs{\nabla v(x)}\leq1-\eps\}$ for $\eps>0$, we obtain
$$1=\abs{\nabla u(x)}\leq \lambda\abs{\nabla v(x)}+(1-\lambda)\abs{\nabla w(x)}\leq \lambda(1-\eps)+(1-\lambda)=1-\lambda \eps,\quad\text{for a.e. }x\in\Omega_\eps.$$
Since $\lambda>0$, this implies that $|\Omega_\eps|=0$ and hence $\abs{\nabla v}=1$ almost everywhere in $\Omega$. Applying the same argument to $w$ shows that $\abs{\nabla w}=1$ holds almost everywhere, as well. Using the Cauchy-Schwarz inequality, we can compute for almost every $x\in\Omega$ 
\begin{align*}
    1&=\abs{\nabla u(x)}^2\\
    &=\lambda^2\abs{\nabla v(x)}^2+(1-\lambda)^2\abs{\nabla w(x)}^2+2\lambda(1-\lambda)\nabla v(x)\cdot\nabla w(x)\\
    &\leq \lambda^2+(1-\lambda)^2+2\lambda(1-\lambda)\\
    &=1.
\end{align*}
Since $\abs{\nabla v}=1=\abs{\nabla w}$, equality has to hold for Cauchy-Schwarz which implies that $\nabla v(x)=c\nabla w(x)$ for some $c\geq 0$. Using that $\abs{\nabla v}=1=\abs{\nabla w}$ implies $c=1$ and hence $\nabla v=\nabla w$ almost everywhere in $\Omega$. Therefore, $v-w$ is constant in $\Omega$ and from $v,w=0$ on $\partial\Omega$ we infer that $v=w$, a contradiction.

For the second inclusion we take some $u\in\extr(B_\calJ)$ and---again aiming for a contradiction---we assume that $\Omega\setminus\Omega_{\max}$ has non empty interior. In this case we set
\begin{align}
    v_\pm(x):=
    \begin{cases}
    u,\quad&x\in\Omega_{\max}\\
    u\pm\phi,\quad&x\in\Omega\setminus\Omega_{\max}.
    \end{cases}
\end{align}
with a function $\phi\neq 0$ to be specified. Obviously, it holds $v_+\neq v_-$ since $\abs{\Omega\setminus\Omega_{\max}}>0$ and furthermore $u=v_+/2+v_-/2$. If we can choose $\phi$ in such a way that $\calJ(v_\pm)\leq 1$, we have reached the desired contradiction. Since $\Omega\setminus\Omega_{\max}$ has non-empty interior there is $\eps>0$ and a set $\Omega_\eps\subset \Omega\setminus\Omega_{\max}$ with non-empty interior such that $\abs{\nabla u}\leq 1-\eps$ almost everywhere on $\Omega_\eps$. If we define 
\begin{align}
    \phi(x)=
    \begin{cases}
    \eps\dist(x,\partial\Omega_\eps),\quad &x\in\Omega_\eps,\\
    0,\quad&\text{else},
    \end{cases}
\end{align}
we infer that $\abs{\nabla v_\pm(x)}=1$ for $x\in\Omega_{\max}$ and $\abs{\nabla v_\pm(x)}\leq (1-\eps)+\eps=1$ for $x\in\Omega\setminus\Omega_{\max}$. Hence, it holds $\calJ(v_\pm)\leq 1$ which means $v_\pm\in B_\calJ$. Finally, $\phi\neq 0$ holds since $\Omega_\eps$ has non-empty interior and therefore does not coincide with its boundary. This is a contradiction and we can conclude.
\end{proof}

\begin{cor}[Distance function is extreme point]
Since $\Omega_{\max}=\Omega$ for the distance function to $\partial\Omega$, we obtain that the distance function is an extreme point.
\end{cor}

\begin{rem}
In general, both inclusions in Proposition~\ref{prop:extreme} are proper. The second inclusion is proper even in one dimension, as Example~\ref{ex:distance_extreme} below shows. In general, also the first inclusion is proper since in \cite{rolewicz1986extremal} the author constructs a extremal function $u:[0,1]^2\to\R$ with $\norm{\nabla u}_\infty=1$ whose gradient is supported on a set with arbitrarily small positive measure. This function can be slightly modified to vanish on the boundary of $\Omega$ and hence provides a valid counterexample. The construction involves the distance function of a fat Cantor set, which we have already investigated in Example~\ref{ex:cantor}, and relies on a connectedness argument. However, in one space dimension one can prove that the first inclusion is indeed an equality. 
\end{rem}
Before we prove that the first inclusion in Proposition~\ref{prop:extreme} is an equality in one dimension, we give an example to show the second inclusion is proper. To this end we show that the distance function to a fat Cantor set is no extreme point.

\begin{example}[Distance function to Smith-Volterra-Cantor set]\label{ex:distance_extreme}
As in Example~\ref{ex:cantor} we let $u(x)=\dist(x,F)$ denote the distance function of the fat Smith-Volterra-Cantor set $F\subset \Omega$ with $\Omega=[0,1]$. Trivially, since $\Omega\setminus \Omega_{\max}=F$, it holds that 
$$u\in\left\lbrace u\in W^{1,\infty}_0(\Omega)\st \calJ(u)=1\wedge \inn\left(\Omega\setminus \Omega_{\max}\right)=\emptyset\right\rbrace $$
but we will show that $u\notin\extr(B_\calJ)$. To this end, let $f=u'$ which is defined almost everywhere and meets $\norm{f}_\infty=1$. We define
\begin{align}
    g_\pm(x):=
    \begin{cases}
    f(x),\quad&x\notin F,\\
    \pm 1,\quad&x\in F\cap\left[0,\frac{1}{2}\right],\\
    \mp 1,\quad&x\in F\cap\left[\frac{1}{2},1\right],
    \end{cases}
\end{align}
and observe that $g_+\neq g_-$ since $F$ has positive measure. Next, we define the functions for almost every $x\in \Omega$ 
\begin{align}
    \tilde{f}(x)&:=\frac{1}{2}g_+(x)+\frac{1}{2}g_-(x)=
\begin{cases}
f(x),\quad&x\notin F,\\
0,\quad&x\in F,
\end{cases}\\
\tilde{u}(x)&:=\int_0^x\tilde{f}(t)\d t.
\end{align}
Using the definition of function $\tilde{f}$ and the fact $\int_a^b f(t)\d t=0$ for every maximally chosen interval $(a,b)\subset \Omega\setminus F$, it is easy to see that $\tilde{u}=u$ holds almost everywhere in $\Omega$. In particular, this also implies that $\tilde{f}=f$ almost everywhere. Finally, we can express $u$ as $u=v_+/2+v_-/2$, where
$$
v_\pm(x):=\int_0^x g_\pm(t)\d t
$$
meet $v_\pm\in W^{1,\infty}_0(\Omega)$ and hence $\calJ(v_\pm)=\norm{v'_\pm}_\infty=\norm{g_\pm}_\infty=1$. This shows that $u$ is no extreme point.
\end{example}
The construction of this example carries over to the general case and allows us to prove that the first inclusion in Proposition~\ref{prop:extreme} is an equality in one space dimension. Note that for Lipschitz continuous functions with one prescribed value in the interval the following was already proved in \cite{rolewicz1984optimal}. However, since we demand zero boundary conditions on both boundary points, the proof changes.
\begin{prop}[Extreme points in one space dimension]
Let $\Omega\subset\R$ be an interval. Then it holds
\begin{align}
    \extr(B_\calJ)=\left\lbrace u\in W^{1,\infty}_0(\Omega)\st \abs{\Omega\setminus\Omega_{\max}}=0\right\rbrace.
\end{align}
\end{prop}
\begin{proof}
We just have to show the inclusion ``$\subset$''. Assume that we have a function $u\in W^{1,\infty}_0(\Omega)$ such that $|\Omega_0|>0$ where $\Omega_0:=\Omega\setminus \Omega_{\max}$. Without loss of generality we assume that $\Omega=[0,1]$. We let $f=u'$ denotes its derivative and since $|\Omega_0|>0$ there is some $\eps\in(0,1]$ such that set $\Omega_\eps:=\{x\in \Omega\st|f(x)|\leq 1-\eps\}$ has positive measure. We define
\begin{align}\label{eq:decomposition_1D_extreme}
    g_\pm(x)&=
    \begin{cases}
    f(x),\quad&x\in \Omega\setminus \Omega_\eps,\\
    f(x)\pm\eps,\quad&x\in \Omega_\eps^1,\\
    f(x)\mp\eps,\quad&x\in \Omega_\eps^2,
    \end{cases}
\end{align}
where the sets $\Omega_\eps^k$ for $k=1,2$ meet $\Omega_\eps=\Omega_\eps^1\,\dot{\cup}\,\Omega_\eps^2$ and are chosen in such a way that $\int_0^1g_\pm(t)\d t=0$. The construction works as follows. For $\alpha\in[0,1]$ we define the continuous function
$$h(\alpha)=\int_{\Omega\setminus \Omega_\eps}f(t)\d t+\int_{\Omega_\eps\cap[0,\alpha]}f(t)+\eps\d t+\int_{\Omega_\eps\cap[\alpha,1]}f(t)-\eps\d t.$$
Since $u$ vanishes on the boundary of $\Omega$ its derivative $f$ has zero mean. Hence, we find that
\begin{align*}
    h(0)&=\int_\Omega f(t)\d t-\eps|\Omega_\eps|=-\eps|\Omega_\eps|<0,\\
    h(1)&=\int_\Omega f(t)\d t+\eps|\Omega_\eps|=\eps|\Omega_\eps|>0.
\end{align*}
Hence, by the intermediate value theorem for continuous functions, there has to be $\tilde{\alpha}\in(0,1)$ such that $h(\tilde{\alpha})=0$. Setting $\Omega_\eps^1:=\Omega_\eps\cap[0,\tilde{\alpha}]$ and $\Omega_\eps^2:=\Omega_\eps\cap(\tilde{\alpha},1]$, we see from \eqref{eq:decomposition_1D_extreme} that $h(\tilde{\alpha})=0$ is equivalent to $\int_0^1g_\pm(t)\d t=0$.

It is obvious that $g_+\neq g_-$ and $\norm{g_\pm}_\infty=1$. Furthermore, it holds $f=g_+/2+g_-/2$ which means that we decompose $u=v_+/2+v_-/2$ where $v_\pm=\int_0^x g_\pm(t)\d t$ meet $\norm{v_\pm'}_\infty=\norm{g_\pm}_\infty=1$ and have zero boundary conditions due to $\int_0^1 g_\pm (t)\d t=0$. Hence it holds $\calJ(v_\pm)=1$ and we can conclude. 
\end{proof}

\section{Extension to finite weighted graphs}
\label{sec:graph}
In this section we analyse a discrete version of functional $\calJ$ within the framework of finite weighted graphs. This requires equipping the graph with suitable differential operators and function space structures, according to~\cite{elmoataz2015p}. The main appeal of differential calculus on graphs is certainly that it allows for complicated topologies, and generalises standard finite difference approximations on grids. Furthermore, graphs do not necessarily have to be interpreted as approximations of physical domains, but can also model images, networks, and databases. 

After introducing notation and important quantities related to finite weighted graphs, we analyse the functional $\calJ_w$, given in~\eqref{eq:fctl_graph} below. In more detail, we study its ground states, characterize its subdifferential and extreme points and investigate some properties of eigenfunctions. On of the main results is Theorem~\ref{thm:ground_states_distance_graph} below, which states that ground states are distance functions, just as in the continuous case. In general, many results carry over from the continuous case directly, which is why we omit most proofs. 

A finite weighted graph $G$ is a triple $G=(V,E,w)$, consisting of a finite set of vertices $V$, an edge set $E\subset V\times V$, and a weight function $w:E\to \R_{\geq 0}$. The notation $x\sim y$ for $x,y\in V$ indicates that $(x,y)\in E$. In the following, we assume the symmetry conditions
\begin{align}
    &x\sim y\iff y\sim x\\
    &w(x,y)=w(y,x),\quad\text{if }x\sim y.
\end{align}
Furthermore, we assume that the graph is connected, which means that for any two vertices $x,y\in V$ there are edges $(x_0,x_1),\,(x_1,x_2),\dots,(x_{n-1},x_n)\in E$ such that $x_0=x$ and $x_n=y$. On the graph we can define vertex functions $\H(V)=\{u:V\to\R\}$ and edge functions $\H(E)=\{q:E\to\R\}$ which can be viewed as real Hilbert spaces with the following inner products
\begin{align}
    \langle u,v\rangle&=\sum_{x\in V} u(x)v(x),\quad u,v\in\H(V),\\
    \langle q,p\rangle&=\sum_{x\sim y}q(x,y)\,p(x,y),\quad q,p\in\H(E).
\end{align}
If an edge function $q\in\H(E)$ meets $q(x,y)=-q(y,x)$ for all $x,y\in V$ we call $q$ anti-symmetric. Next, we define the weighted gradient $\nablaw$ of a vertex function $u\in \H(V)$ evaluated on an edge $(x,y)\in E$ as
\begin{align}
    (\nablaw u)(x,y):={w(x,y)}^{\frac{1}{2}}(u(y)-u(x)),
\end{align}
which makes $\nabla_wu:E\to\R$ an anti-symmetric edge function. Obviously, $\nabla_w:\H(V)\to\H(E)$ is a linear operator and its adjoint is given by $\nablaw^*=-\divw$, where 
\begin{align}
    (\divw q)(x):=\sum_{y\st x\sim y}{w(x,y)}^\frac{1}{2}(q(y,x)-q(x,y))
\end{align}
denotes the weighted divergence of an edge function $q\in\H(E)$ evaluated in $x_i\in V$. This implies the validity of the integration by parts formula 
\begin{align}
    \langle q,\nablaw u\rangle=-\langle\divw q,u\rangle,\quad\forall u\in\H(V),\,q\in\H(E).
\end{align}
Furthermore, we define the one-sided gradient
\begin{align}
    (\nablaw^- u)(x,y):={w(x,y)}^{\frac{1}{2}}(u(y)-u(x))_-,
\end{align}
where $(x)_-:=-\min(x,0)$ and introduce $p$-norms on $\H(V)$ and $\H(E)$ by setting
\begin{align}
    \norm{u}_{p\phantom{\infty}} &\ = \ \left(\sum_{x\in V}|u(x)|^p\right)^\frac{1}{p},\quad 1\leq p<\infty,\\
    \norm{u}_{\infty\phantom{p}} &\ = \ \max_{x\in V}|u(x)|,\\
    \norm{q}_{p\phantom{\infty}} &\ = \ \left(\sum_{x\sim y}|q(x,y)|^p\right)^\frac{1}{p},\quad 1\leq p<\infty,\\
    \norm{q}_{\infty\phantom{p}} &\ = \ \max_{x\sim y}|q(x,y)|.
\end{align}
Next we take a subset of the vertex set $\Gamma\subset V$ which we identify with a Dirichlet boundary, and consider the subspace $\H_0(V)=\{u\in\H(V)\st u(x)=0,\,\forall x\in\Gamma \}$ of all vertex functions which vanish on $\Gamma$. Analogous to \eqref{eq:fctl_v2}, we define the functional
\begin{align}\label{eq:fctl_graph}
    \calJ_w(u)=
    \begin{cases}
        \norm{\nablaw u}_{\infty},\quad&u\in\H_0(V),\\
        +\infty,\quad&\text{else}.
    \end{cases}
\end{align}
Note that also $\calJ_w$ is a convex and absolutely one-homogeneous functional on a Hilbert space. The aim of the following section is to analyse $\calJ_w$ and show analogous results as we have seen in Section~\ref{sec:analysis_fctl}. 
\subsection{Ground states and properties of the distance function}
\label{sec:graph_dist}
First we will study ground states of $\calJ_w$, i.e., functions $u^*\in\H_0(V)$ such that
\begin{align}
    u^*\in\argmin_{u\in\H_0(V)}\frac{\calJ_w(u)}{\norm{u}_2}.
\end{align}
Since ground states are invariant under multiplication with scalars we can again replace the problem with
\begin{align}\label{eq:ground_state_graph}
    u^*\in\argmax_{\substack{u\in\H_0(V)\\|\nablaw u|\leq 1}}\norm{u}_2.
\end{align}

\begin{thm}[Ground states are distance functions]\label{thm:ground_states_distance_graph}
Up to global sign, the unique solution of \eqref{eq:ground_state_graph} is given by 
\begin{align}\label{eq:graph_distance_fct}
    u^*(x)=d(x):=\min_{y\in\Gamma}d_w(x,y),\quad x\in V,
\end{align}
where 
\begin{align}
    d_w(x,y):=\min\left\lbrace\sum_{i=1}^n w(x_{i-1},x_i)^{-\frac{1}{2}}\st n\in\N,\;x_0\sim\dots\sim x_n,\;x_0=x,x_n=y\right\rbrace
\end{align}
denotes the graph distance of $x,y\in V$.
\end{thm}
\begin{proof}
Since $d_w(\cdot,\cdot)$ is a distance and hence fulfills the triangle inequality is is standard to check that \eqref{eq:graph_distance_fct} is 1-Lipschitz and hence admissible in \eqref{eq:ground_state_graph}. To show that \eqref{eq:graph_distance_fct} indeed solves \eqref{eq:ground_state_graph} we note that by possibly replacing $u^*$ with $|u^*|$ one can restrict the maximization to non-negative functions. From there it is straightforward to see that $u(x)\leq d(x)$ for all $x\in V$ which implies that \eqref{eq:graph_distance_fct} solves \eqref{eq:ground_state_graph}. 
\end{proof}

Note that on graphs the distance function, and hence the solution of \eqref{eq:ground_state_graph}, does typically not fulfill $|(\nablaw d)(x,y)|=1$ for all $(x,y)\in E$, as the following simple example shows.

\begin{example}[Distance function with vanishing gradient]\label{ex:distance_on_graph}
We consider the graph $G=(V,E)$ with vertices $V=\{x_0,x_1,x_2,x_3\}$ and edges $E=\{(x_0,x_1),(x_1,x_2),(x_2,x_3)\}$. The weights are assumed to be one and we take $\Gamma=\{x_0,x_3\}$. Using compact tuple notation, the distance function is given by
$$
d=(0,1,1,0)
$$
and obviously it holds $(\nablaw d)(x_1,x_2) = 0$.
\end{example}

Of course, the fact that $|(\nablaw d)(x,y)|\neq 1$ in general, is due to the fact that $(\nablaw d)(x,y)$ can only be interpreted as directional derivative and not as full gradient. However, we have the following theorem.

\begin{prop}[Properties of the distance function]\label{prop:prop_dist_graph}
For all $x\in V$ and $y\sim x$ the distance function $d$ to $\Gamma$ meets
\begin{align}\label{eq:dir_deriv_dist}
    |(\nablaw d)(x,y)|&
    \begin{cases}
    =1,\quad\text{if }y\in\mathrm{SP}(x,\Gamma)\text{ or }x\in\mathrm{SP}(y,\Gamma)\\
    <1,\quad\text{else},
    \end{cases}
\end{align}
where 
\begin{align}
    \mathrm{SP}(x,\Gamma):=\left\{x_0\sim\dots\sim x_n,\; x_0=x,\;x_n\in\Gamma,\;d(x)=\sum_{i=1}^nw(x_{i-1},x_i)^{-\frac{1}{2}} \right\}
\end{align}
denotes the set of all shortest paths from $x$ to $\Gamma$.
\end{prop}
\begin{proof}
Let $x\in V$ and $y\sim x$ be a neighboring node. If $y\in SP(x,\Gamma)$, then $x\sim y\sim x_1\sim\dots\sim x_n$ with $x_n\in\Gamma$ is a shortest path for $x$ and $y\sim x_1\sim\dots\sim x_n$ a shortest path for $y$. Consequently, $d(x)$ and $d(y)$ differ by the value $d_w(x,y)=w(x,y)^{-\frac{1}{2}}$ which means $|(\nablaw d)(x,y)|=1$. If $x\in\mathrm{SP}(y,\Gamma)$ the same holds true by interchanging the roles of $x$ and $y$. 

In the case that $x$ and $y$ do not lie on a common shortest path it holds 
\begin{align*}
    d(y)&<d(x)+w(x,y)^{-\frac{1}{2}},\\
    d(x)&<d(y)+w(x,y)^{-\frac{1}{2}},
\end{align*}
and hence $|d(y)-d(x)|<w(x,y)^{-\frac{1}{2}}$, which implies $|(\nablaw d)(x,y)|<1$. 
\end{proof}
For a non-weighted graph, meaning that all weights are one, we can obtain a more precise characterization of the directional derivatives of the distance function. Furthermore, we show that the 1-norm of the one-sided gradient $\nablaw^-d$ as in a point $x\in V$ counts the number of optimal paths from $x$ to $\Gamma$.
\begin{cor}[Unitary weights]
Assume that $w(x,y)=1$ for all $x\sim y$. Then for all $x\in V$ and $y\sim x$ it holds
\begin{align}
    |(\nablaw d)(x,y)|=
     \begin{cases}
        1,\quad\text{if }y\in\mathrm{SP}(x,\Gamma)\text{ or }x\in\mathrm{SP}(y,\Gamma),\\
        0,\quad\text{else}.
    \end{cases}
\end{align}
Furthermore, it holds
\begin{align}\label{eq:one_sided_grad_dist}
    \sum_{x\sim y}|(\nablaw^-)d(x,y)|&=\#\{y\st y\in \mathrm{SP}(x,\Gamma)\}.
\end{align}
\end{cor}
\begin{proof}
The first statement follows from Proposition~\ref{prop:prop_dist_graph}, observing that $1>|(\nablaw d)(x,y)|=|d(y)-d(x)|$ implies $d(x)=d(y)$ since $d$ takes only integer values. For the second statement we note that the one-sided gradient $(\nablaw^-d)(x,y)$ equals zero if $x\in \mathrm{SP}(y,\Gamma)$ since in this case $d(y)>d(x)$. Hence, it holds
\begin{align}
    |(\nablaw^- d)(x,y)|=
     \begin{cases}
        1,\quad\text{if }y\in\mathrm{SP}(x,\Gamma),\\
        0,\quad\text{else}.
    \end{cases}
\end{align}
which directly implies \eqref{eq:one_sided_grad_dist}.
\end{proof}

\subsection{Subdifferential and eigenfunctions}
\label{sec:graph_properties}
After having characterized the ground state of $\calJ_w$ as distance function and having studied its geometric properties, we proceed with the characterization of the subdifferential $\partial\calJ_w$ and study properties of eigenfunctions.

In the following, we fix a function $u\in\H_0(V)$, and define the set of edges where the gradient of $u$ attains its maximal modulus as
\begin{align}\label{eq:E_max}
    E_{\max}=\left\lbrace(x,y)\in E\st |(\nablaw u)(x,y)|=\calJ_w(u)\right\rbrace.
\end{align}
Note that $E_{\max}$ is never empty due to the finite dimensional nature of all quantities involved. The following proposition characterizes the subdifferential of $\calJ_w$ analogously to Proposition~\ref{prop:characterization}.

\begin{prop}[Characterization of the subdifferential]\label{prop:characterization_graph}
Let $u\in\H_0(V)\setminus\{0\}$ and let $E_{\max}$ be given by \eqref{eq:E_max}. Then it holds
\begin{align*}
    \partial\calJ_w(u)=\lbrace-\divw q\st q\in\H(E),\;\norm{q}_1=1,\;q(x,y)=0\;\forall (x,y)\in E\setminus E_{\max},\\
    q(x,y)(\nablaw u)(x,y)=|q(x,y)|\,|(\nablaw u)(x,y)|\;\forall (x,y)\in E_{\max}\rbrace.
\end{align*}
\end{prop}
Next we study extreme points of the unit ball $B_{\calJ_w}$ of $\calJ_w$, given by
\begin{align}
    B_{\calJ_w}=\left\lbrace u\in\H(V)\st\calJ_w(u)\leq 1\right\rbrace.
\end{align}

Next we turn to the study of eigenfunctions of $\partial\calJ_w$.
We should first remark that $\lambda u\in\partial\calJ_w(u)$ is not a good definition for eigenfunctions due to the Dirichlet conditions on $\Gamma$. This means that in general, one cannot find $u\in\H_0(V)$ and $q\in\H(E)$ such that $\lambda u=-\divw q$. This is illustrated in the following example.
\begin{example}
Let $V=\{x_0,x_1,x_2\}$, $E=\{(x_0,x_1),(x_1,x_2)\}$, and assume all weights are one. We set $\Gamma=\{x_0,x_2\}$. Then, trivially, the distance function $d=(0,1,0)$ is an eigenfunction. If we assume that $\lambda u=-\divw q\in\partial\calJ_w(u)$ then $d(x_0)=0$ implies $q(x_0,x_1)=q(x_1,x_0)$ by definition of the divergence operator. The characterization of the subdifferential Proposition~\ref{prop:characterization_graph} then tells us that $q(x_0,x_1)=0=q(x_1,x_0)$ since $q$ has to be parallel to $(\nablaw d)(x_0,x_1)=1$ and $(\nablaw d)(x_1,x_0)=-1$. The same holds for $q(x_1,x_2)$ and hence $q=0$ which contradicts $-\div q=\lambda d$.
\end{example}

\begin{definition}[Eigenfunctions of $\partial\calJ_w$]
We call $u\in\H_0(V)$ an eigenfunction of $\partial\calJ_w$ if there exist $\lambda>0$, and $q\in H(E)$ with $-\div q\in\partial J(u)$, such that
\begin{align}
    \langle\lambda u,v\rangle=\langle-\divw q,v\rangle,\quad\forall v\in\H_0(V).
\end{align}
This is equivalent to $\lambda u(x)=-\divw q(x)$ for all $x\in V\setminus\Gamma$.
\end{definition}

The next example shows that non-negative eigenfunctions of $\partial\calJ_w$ are not unique, opposed to the continuum case where Proposition~\ref{prop:positivity_efs} asserted that every non-negative eigenfunction is a ground state.

\begin{example}[Multiple non-negative eigenfunctions]
We return to the graph from Example~\ref{ex:distance_on_graph}. Functional $\calJ_w$ can be explicitly expressed as
$$\calJ_w(u)=\max(|u_1|,|u_2|,|u_1-u_2|)$$
where $u_i:=u(x_i)$ for $i=1,2$. The unit ball and dual unit ball of $\calJ_w$ are depicted in Figure~\ref{fig:eigenvectors}. Following \cite{bungert2019nonlinear}, eigenvectors are precisely all multiples of vectors in the dual unit ball whose orthogonal hyperplane is tangent to the boundary. Here they correspond to all multiples of the four vertex functions having the values
\begin{align*}
    (0,1/2,1/2,0),\quad 
    (0,1,0,0),\quad
    (0,0,1,0),\quad
    (0,-1,1,0).
\end{align*}
Note that, the first three are also extreme points of the primal unit ball (up to scalar multiplication), whereas the fourth one, marked in red, is not. Furthermore, the first three eigenfunctions are all non-negative.
\input{eigenvectors.tex}

\end{example}

We have just seen that non-negative eigenfunction do in general not coincide with a ground state, as it is the case in the continuum. However, thanks to the following proposition, whose proof works just as in the continuous case of Proposition~\ref{prop:positivity_efs}, \emph{positive} eigenfunctions are unique.
\begin{prop}[Positive eigenfunctions]
Let $u\in\H_0(V)$ be a non-negative eigenfunction with $\calJ_w(u)=1$ and let $d$ denote the distance function to $\Gamma$. Then for every $x\in V$ it holds $u(x)=d(x)$ or $u(x)=0$. Consequently, any eigenfunction which is positive in $V\setminus\Gamma$ coincides with a ground state.
\end{prop}

\subsection{Extreme points}
As in the continuous case of Section~\ref{sec:extreme_points}, the main motivation for studying extreme points are representer theorems. They assert certain optimization problems involving $\calJ_w$ admit a solution which is a linear combination of extreme points. As before, we obtain a characterization of extreme points which is based on the existence of paths from every vertex to the boundary $\Gamma$ such that all directional derivatives are one along this path.
\begin{thm}[Characterization of extreme points]\label{thm:extremal_points_graph}
It holds that
\begin{align*}
    \extr(B_{\calJ_w})=\{u\in\H_0(V)\st\forall x\in V\;\exists x_0\sim\dots\sim x_n\;x_0=x,\;x_n\in\Gamma,\\
    \qquad\;|(\nablaw u)(x_{i-1},x_i)|=1,\;\forall i=1,\dots,n\}.
\end{align*}
\end{thm}
However, as opposed to the continuous case, even one-dimensional extreme functions do not necessarily have constant modulus of the gradient, as the following example shows.
\begin{example}
We return to Example~\ref{ex:distance_on_graph} with the distance function $d(x)=(0,1,1,0)$, which fulfills $\nablaw d(x_1,x_2)=0$. Nevertheless, it obviously is an extreme point taking the paths $x_1\sim x_0$ and $x_2\sim x_4$. If one however adds a node $x_4$ with $x_3\sim x_4$ and sets $u(x)=(0,1,1,0,0)$ this is not extreme anymore, since there is no path from $x_3$ to $x_0$ or $x_4$ along which $\nablaw u$ has modulus one.
\end{example}

\section*{Acknowledgement}
This work was supported by the European Unions Horizon 2020 research and innovation
programme under the Marie Sk{\l}odowska-Curie grant agreement No 777826 (NoMADS).

YK is supported by the Royal Society (Newton International Fellowship NF170045 Quantifying Uncertainty in Model-Based Data Inference Using Partial Order) and the Cantab Capital Institute for the Mathematics of Information.
\bibliographystyle{abbrv}
\bibliography{bibliography}

\begin{appendix}
\section{Proof of Proposition~\ref{prop:characterization_integral}}
Before we proceed to the proof of the theorem, we need a straightforward approximation lemma for Lipschitz functions.
\begin{lemma}\label{lem:approximation}
Let $v\in W^{1,\infty}_0(\Omega)$. Then there exists a sequence $(v_n)\subset C^\infty_c(\Omega)$ such that 
\begin{itemize}
    \item $\norm{\nabla v_n}_\infty\leq\norm{\nabla v}_\infty$
    \item $\norm{v-v_n}_\infty\to 0$ as $n\to\infty$
\end{itemize}
\end{lemma} 
\begin{proof}
First, we approximate $v$ with compactly supported functions $(w_n)\subset C^{0,1}_c(\Omega)$. To this end, set 
\begin{align*}
    w_n^\pm(x)=\min(v^\pm(x)-1/n,0)
\end{align*}
where $v^\pm$ denotes the positive and negative part of $v$. 
If we define $w_n:=w_n^+-w_n^-$ it holds
$$\norm{v-w_n}_\infty\leq 1/n\to 0,\quad n\to\infty.$$
and $\norm{\nabla w_n}_\infty\leq\norm{\nabla v}_\infty$.
Furthermore, all $w_n$ are compactly supported.
To see this one notes that
$$|v(x)|\leq\calJ(v)\dist(x,\partial\Omega),$$
which implies that $w_n=0$ for all $x\in\Omega$ such that $\dist(x,\partial\Omega)\leq{1}/(\calJ(v)n)$.
Now let $\eps=1/(2n)$ and define mollifications $v_n:=w_n\ast\varphi_\eps$. 
Then it holds $\norm{\nabla v_n}_\infty\leq\norm{\nabla w_n}_\infty\leq\norm{\nabla v}_\infty$ and
$$\norm{v-v_n}_\infty\leq\norm{v-w_n}_\infty+\norm{v_n-w_n}_\infty.$$
The first term on the right hand side can be bounded by $1/n$ as shown above. For the second term we notice
$$\left|v_n(x)-w_n(x)\right|\leq\int_\Omega|\varphi_\eps(y)|w_n(x-y)-w_n(x)|\d y\leq \norm{\nabla w_n}_\infty\frac{1}{2n}\leq\norm{\nabla v}_\infty\frac{1}{2n},\quad\forall x\in\Omega.$$
Hence, both terms converge to zero and we can conclude.
\end{proof}
\begin{proof}[Proof of Proposition~\ref{prop:characterization_integral}]
We follow the argumentation of \cite[Prop.~7]{bredies2016pointwise} who deal with the subdifferential of the total variation. 
Defining the set
$$C:=\{-\div q\st q\in C^\infty(\overline{\Omega},\R^n),\;\norm{q}_1\leq 1\}$$
it holds $\calJ(u)=\chi_C^*(u)$, where $\chi$ denotes the characteristic function of a set and $.^\ast$ is the convex conjugate.
Hence, it holds $\calJ^*(\sg)=\chi_C^{**}(\sg)=\chi_{\overline{C}}(\sg)$ and by~\eqref{eq:subdifferential} one gets that $\sg\in\partial\calJ(u)$ if and only if $\sg\in\overline{C}$ and $\langle\sg,u\rangle=\calJ(u)$.

Therefore, we just have to find the $L^2$-closure of $C$ and we claim it holds
$$\overline{C}=\{-\div q\st g=g+r,\,g\in G^1_0(\Omega),\,r\in\calN(\div;\Omega),\,\abs{q}(\Omega)\leq 1\}=:K.$$
\paragraph{Inclusion $K\subset\overline{C}$:}
For this inclusion it is enough to show that for any $q\in \calM(\Omega,\R^n)$ with $-\div q\in K$ it holds 
$$\int_\Omega-(\div q)v\dx\leq\calJ(v),\quad\forall v\in L^2(\Omega)$$
since this implies $\chi_{\overline{C}}(-\div q)=\calJ^*(-\div q)=0$ and hence $-\div q\in\overline{C}$.
Indeed, it suffices to check the inequality for $v\in W^{1,\infty}_0(\Omega)$. 
By Lemma~\ref{lem:approximation}, we can find a sequence functions $(v_n)\subset C^\infty_c(\Omega)$ such that $\norm{\nabla v_n}_\infty\leq \norm{\nabla v}_\infty$ and $\norm{v_n-v}_\infty\to 0$ as $n\to\infty$.
This implies
\begin{align*}
    \int_\Omega-(\div q)\, v\dx=\lim_{n\to\infty}\int_\Omega-(\div q)\,v_n\dx=\lim_{n\to\infty}\int_\Omega \nabla v_n\cdot\d q\dx\leq \abs{q}(\Omega)\norm{\nabla v_n}_\infty\leq\calJ(v).
\end{align*}
\paragraph{Inclusion $\overline{C}\subset K$:}
To prove the converse inclusion it suffices to show that $K$ is closed in $L^2(\Omega)$ since $C\subset K$ is obviously correct.
Let $(q_n)\subset \calM(\Omega,\R^n)$ be a sequence of measure such that $q_n=g_n+r_n$ with $(g_n)\subset G^1_0(\Omega)$, $(r_n)\subset\calN(\div;\Omega)$.
Furthermore, assume that $\abs{q_n}(\Omega)\leq 1$ and $-\div q_n\to\mu$ strongly in $L^2(\Omega)$.
From \cite[(1.2)]{auchmuty2006divergence} we infer that $\norm{g_n}_2$ is uniformly bounded and hence, up to a subsequence, $g_n$ converges weakly in $L^2(\Omega)$ to some $g\in L^2(\Omega)$. 
By the closedness of $G^1_0(\Omega)$ we infer that $g\in G^1_0(\Omega)$. 
We first show that $\mu=-\div g$.
To this end, we use the convergences $g_n\rightharpoonup g$ and $\div q_n\to\mu$ together with the fact that $\div g_n=\div q_n$ to compute
\begin{align*}
    \langle g,\nabla\varphi\rangle=\lim_{n\to\infty}\langle g_n,\nabla\varphi\rangle=-\lim_{n\to\infty}\langle\div g_n,\varphi\rangle=-\lim_{n\to\infty}\langle\div q_n,\varphi\rangle=\langle\mu,\varphi\rangle,\quad\forall\varphi\in C^\infty_c(\Omega),
\end{align*}
which shows $\mu=-\div g$.
Since $\abs{q_n}(\Omega)\leq 1$, by the sequential Banach-Alaoglu theorem there exists a measure $q\in\calM(\Omega,\R^n)$ such that, up to a subsequence, it holds $q_n\rightharpoonup q$.
The lower semi-continuity of the total variation implies $\abs{q}(\Omega)\leq 1$.
Furthermore, $g_n\rightharpoonup g$ implies that in fact $r_n\rightharpoonup r:=q-g$. 
By the closedness of $\calN(\div;\Omega)$, we infer $r\in\calN(\div;\Omega)$. 
Hence, we have shown that $\mu= -\div q\in K$, as desired.
\end{proof}

\section{Proof of Theorem~\ref{thm:extremal_points}}
In order to prove the theorem, we first need the following lemma which states a triangle inequality for the map $x\mapsto \eps_{x,z}^u$, given by \eqref{eq:def_eps}.
\begin{lemma}\label{lem:triangle_eps}
Let $u\in B_\calJ$, $x,y\in\Omega$, and $z\in\partial\Omega$. Then it holds
\begin{align*}
    \eps_{y,z}^u\leq\eps_{x,z}^u+|x-y|-|u(x)-u(y)|.
\end{align*}
\end{lemma}
\begin{proof}
We denote by $(\eps^n)_{n\in\N}$ a minimizing sequence  for $\eps_{x,z}^u$, i.e. $\lim_{n\to\infty}\eps^n=\eps_{x,z}^u$. This means that for each $n\in\N$  there exists a path of $n$ points $x_0^n=z,x_1^n,\dots,x_n^n=x$ connecting $z$ and $x$, and non-negative numbers $(\eps_i)_{i=1,\dots,n}$ such that 
\begin{align*}
    |x_{i-1}-x_i|-\eps_i&\leq|u(x_{i-1})-u(x_i)|,\quad i=1,\dots,n,\\
    \sum_{i=1}^n\eps_i&\leq\eps^n.
\end{align*}
Now we define the path of $n+1$ points
$$y_i=
\begin{cases}
x_i^n,\quad &i=0,\dots,n,\\
y,\quad &i=n+1,
\end{cases}
$$
which connects $z$ and $y$, set $\eps_{n+1}=|x-y|-|u(x)-u(y)|\geq 0$, and observe that this constellation is admissible for the minimization that defines $\eps^u_{y,z}$ since
\begin{align*}
    |x_{i-1}-x_i|-\eps_i&\leq|u(x_{i-1})-u(x_i)|,\quad i=1,\dots,n+1,\\
    \sum_{i=0}^{n+1}\eps_i&\leq\eps^n+|x-y|-|u(x)-u(y)|.
\end{align*}
Hence it holds 
$$\eps_{x,y}^u\leq \eps^n+|x-y|-|u(x)-u(y)|$$
and letting $n$ tend to infinity we obtain the desired inequality.
\end{proof}
Now we can proceed to the proof of the theorem.
\begin{proof}[Proof of Thm.~\ref{thm:extremal_points}]
The proof works similar to \cite{farmer1994extreme} with the main difference being that there the point $z=0$ is fixed. Since this causes non-trivial modifications, we present the full proof for completeness.

We start with the implication ``$\impliedby$'': to this end, we assume that \eqref{eq:cons_inf_eps} holds for almost all $x\in\Omega$. Since $\eps_{x,z}^u$ depends continuously on $z\in\partial\Omega$ and $\partial\Omega$ is compact, we infer that for almost all $x\in\Omega$ there exists $z\in\partial\Omega$ with $\eps_{x,z}^u=0$. Aiming for a contradiction we assume $u=v/2+w/2$ with two functions $v,w\in B_\calJ$. Since $\eps_{x,z}^u=0$, for any $\eps>0$ we can find finite sequences of points $(x_i)_{i=0,\dots,n}$ and numbers $(\eps_i)_{i=1,\dots,n}$ satisfying the restrictions such that
$$|x_{i-1}-x_i|-\eps_i\leq |u(x_{i-1})-u(x_i)|,\quad\forall i=1,\dots,n.$$
Without loss of generality we assume that $u(x_{i-1})-u(x_i)\geq 0$. Using also $u=v/2+w/2$ we infer
\begin{align*}
    -\eps_i&=|x_{i-1}-x_i|-\eps_i-|x_{i-1}-x_i|\\
    &\leq|u(x_{i-1})-u(x_i)|-|v(x_{i-1})-v(x_i)|\\
    &\leq u(x_{i-1})-u(x_i)-(v(x_{i-1})-v(x_i))\\
    &=w(x_{i-1})-w(x_i)-(u(x_{i-1})-u(x_i))\\
    &\leq |x_{i-1}-x_i|-(\eps_i-|x_{i-1}-x_i|)\\
    &=\eps_i,
\end{align*}
which means
\begin{align*}
    |u(x_{i-1})-u(x_i)-(v(x_{i-1})-v(x_i))|\leq\eps_i,\quad\forall i=1,\dots,n.
\end{align*}
Iterating this estimate, we obtain
\begin{align*}
    |u(x)-v(x)|&=|u(x_n)-v(x_n)|\\
    &=|u(x_n)-u(x_{n-1})-(v(x_n)-v(x_{n-1}))+u(x_{n-1})-v(x_{n-1})|\\
    &\leq \eps_n+|u(x_{n-1})-v(x_{n-1})|\\
    &\leq\dots\\
    &\leq\sum_{i=1}^n\eps_i+|u(x_0)-v(x_0)|\leq\eps,
\end{align*}
where we used that $x_0=z\in\partial\Omega$ and hence $u(x_0)=v(x_0)=0$ there. Since this estimate holds for all $\eps>0$ and almost all $x\in\Omega$ we infer $u=v$ and hence also $u=w$ in almost everywhere in $\Omega$, which means that $u$ is extreme.

For the converse implication ``$\implies$'' we assume that there exists a set $A\subset\Omega$ of positive measure such that it holds $\hat{\eps}_x:=\inf_{z\in\partial\Omega}\eps_{x,z}^u>0$ for almost all $x\in A$. We define the functions
\begin{align*}
    v_\pm(x)=
    \begin{cases}
        u(x)\pm\hat{\eps}_x,\quad&x\in A,\\
        u(x),\quad&x\in\Omega\setminus A,
    \end{cases}
\end{align*}
which obviously meet $v_+\neq v_-$ and $v_+/2+v_-/2=u$. It remains to show that $v_\pm\in B_\calJ$ to obtain that $u$ is not extreme. We consider $v_+$ only since the considerations for $v_-$ are identical. 
We just have to show that $|v_+(x)-v_-(y)|\leq |x-y|$ for all $x,y\in\Omega$. For $x,y\in\Omega\setminus A$ this is clear and hence we first assume that $x\in\Omega\setminus A$ and $y\in A$. In this case it holds
$$|v_+(x)-v_+(y)|=|u(x)-u(y)-\hat{\eps}_y|\leq|u(x)-u(y)|+\hat{\eps}_y.$$
Since $\hat{\eps}_x=0$ by the assumption $x\in\Omega\setminus A$ we can choose $z_0\in\partial\Omega$ such that $\eps_{x,z_0}^u=0$. By the definition of $\hat{\eps}_y$ and the triangle inequality from Lemma~\ref{lem:triangle_eps} we obtain
$$\hat{\eps}_y\leq\eps_{y,z_0}^u\leq \underbrace{\eps^u_{x,z_0}}_{=0}+|x-y|-|u(x)-u(y)|,$$
which yields
$$|v_+(x)-v_-(y)|\leq |x-y|.$$

Assume now that $x,y\in A$ in which case it holds
\begin{align*}
    |v_+(x)-v_+(y)|=|u(x)-u(y)+\hat{\eps}_x-\hat{\eps}_y|\leq|u(x)-u(y)|+|\hat{\eps}_x-\hat{\eps}_y|.
\end{align*}
Now we choose elements $z_x,z_y\in\partial\Omega$ such that $\hat{\eps}_x=\eps_{x,z_x}^u$ and $\hat{\eps}_y=\eps_{y,z_y}^u$. By using the triangle inequality from Lemma~\ref{lem:triangle_eps} for $z\in\{z_x,z_y\}$ we obtain
$$|u(x)-u(y)|\leq |x-y|+\frac{1}{2}(\eps_{x,z_x}+\eps_{x,z_y})-\frac{1}{2}(\eps_{y,z_x}+\eps_{y,z_y}).$$
After possibly exchanging the roles of $x$ and $y$ we can assume that the right hand side is smaller or equal than $|x-y|$ which concludes the proof.
\end{proof}
\end{appendix}
\end{document}

%% file: eigenvectors.tex
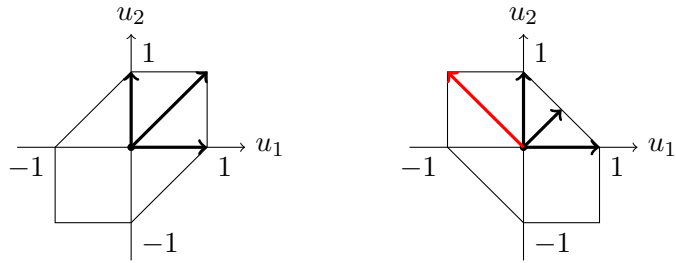
\begin{figure}[h!]
\centering
\begin{tikzpicture}[scale=1.0]
		\node at (0,0) [circle,fill,inner sep=1pt]{};
		\draw[->] (-1.5,0)--(1.5,0) node[right] {$u_1$};
		\draw[->] (0,-1.5)--(0,1.5) node[above] {$u_2$};
		\draw (1,0) node [anchor=north west] {$1$};
	    \draw (0,1) node [anchor=south west] {$1$};     
	    \draw (0,-1) node [anchor=north west] {$-1$};
	    \draw (-1,0) node [anchor=north east] {$-1$};
        \draw (1,0)--(1,1)--(0,1)--(-1,0)--(-1,-1)--(0,-1)--(1,0);
        \draw[->,very thick](0,0)--(1,0);
        \draw[->,very thick](0,0)--(1,1);
        \draw[->,very thick](0,0)--(0,1);
\end{tikzpicture}
\hspace{1cm}
\begin{tikzpicture}[scale=1.0]
        \node at (0,0) [circle,fill,inner sep=1pt]{};
		\draw[->] (-1.5,0)--(1.5,0) node[right] {$u_1$};
		\draw[->] (0,-1.5)--(0,1.5) node[above] {$u_2$};
		\draw (1,0) node [anchor=north west] {$1$};
	    \draw (0,1) node [anchor=south west] {$1$};     
	    \draw (0,-1) node [anchor=north west] {$-1$};
	    \draw (-1,0) node [anchor=north east] {$-1$};
        \draw (1,0)--(0,1)--(-1,1)--(-1,0)--(-1,0)--(0,-1)--(1,-1)--(1,0);
        \draw[->,very thick](0,0)--(1,0);
        \draw[->,very thick](0,0)--(.5,.5);
        \draw[->,very thick](0,0)--(0,1);
        \draw[->,very thick, red](0,0)--(-1,1);
\end{tikzpicture}
\caption{Primal and dual unit balls of $\calJ_w$ with all extreme points and eigenvectors (up to scalar multiplication).\label{fig:eigenvectors}}
\end{figure}